\documentclass{amsart}

\usepackage[english]{babel}
\usepackage{amsmath}
\usepackage{amsthm}
\usepackage{amssymb}
\usepackage{amscd}
\usepackage{xypic}
\usepackage[latin1]{inputenc}

\title[Exact categories, triangulated adjoints and model structures]
{On exact categories and applications to triangulated adjoints and model structures}

\author{Manuel Saor\'\i{}n}
\address{
Departamento de Matem\'aticas\\
Universidad de Murcia, Aptdo. 4021\\
30100 Espinardo, Murcia\\
SPAIN}
\email{msaorinc@um.es}

\author{Jan \v{S}\v{t}ov\'\i\v{c}ek}
\address{
Charles University in Prague, Faculty of Mathematics and Physics\\
Department of Algebra \\
Sokolovska 83, 186 75 Praha 8\\
CZECH REPUBLIC
}
\email{stovicek@karlin.mff.cuni.cz}


\date{\today}

\thanks{%
The first named author is supported by research projects from the D.G.I. of the Ministerio de Ciencia e Innovaci\'{o}n and the  Fundaci\'{o}n `S\'{e}neca' of Murcia (Grupos de Excelencia), with a part of FEDER funds. Part of his contribution to the paper was done  while he was on sabbatical leave at the Universit\'{e} de Paris 7 supported by a grant from the same Ministerio. He thanks all these institutions for their support. The second named author is supported by the research project MSM~0021620839 and by the grant GA\v{C}R P201/10/P084.
}

\theoremstyle{plain}
\newtheorem{theor}{Theorem}[section]
\newtheorem{lemma}[theor]{Lemma}
\newtheorem{prop}[theor]{Proposition}
\newtheorem{cor}[theor]{Corollary}

\theoremstyle{definition}
\newtheorem{defi}[theor]{Definition}
\newtheorem{nota}[theor]{Notation}
\newtheorem{setup}[theor]{Setup}

\theoremstyle{remark}
\newtheorem{rem}[theor]{Remark}

\newtheorem{warn}[theor]{Warning}
\newtheorem{exem}[theor]{Example}
\newtheorem{exems}[theor]{Examples}

\renewcommand{\iff}{if and only if }
\newcommand{\st}{such that }
\newcommand{\wrt}{with respect to }

\def\Hom{\mathop{\rm Hom}\nolimits}
\def\stHom{\mathop{\underline{\rm Hom}}\nolimits}
\def\End {\mathop{\rm End}\nolimits}
\def\Ext {\mathop{\rm Ext}\nolimits}
\def\Ker {\mathop{\rm Ker}\nolimits}
\def\Im {\mathop{\rm Im}\nolimits}
\def\Coker {\mathop{\rm Coker}\nolimits}

\def\cf {\mathop{\rm cf}\nolimits}

\newcommand{\Mod}[1]{\hbox{\rm Mod-}{#1}}

\newcommand{\Proj}[1]{\hbox{\rm Proj-}{#1}}
\newcommand{\Flat}[1]{\hbox{\rm Flat-}{#1}}

\newcommand{\Qco}[1]{\mathfrak{Qco}({#1})}
\newcommand{\coh}[1]{\mathfrak{coh}({#1})}
\newcommand{\Vect}[1]{\hbox{\rm Vect-}{#1}}
\newcommand{\vect}[1]{\hbox{\rm vect-}{#1}}

\newcommand{\Filt}[1]{\hbox{\rm Filt-}{#1}}
\newcommand{\dg}[1]{\hbox{\rm dg-}\tilde{#1}}
\newcommand{\codg}[1]{\hbox{\rm codg-}\tilde{#1}}

\newcommand{\Iinj}{\mathcal{I}\hbox{\rm-inj}}

\newcommand{\ICells}{\mathcal{I}\hbox{\rm-Cells}}

\newcommand{\Cpx}[1]{\mathbf{C}({#1})}

\newcommand{\Cplus}[1]{\mathbf{C}^{+}({#1})}

\newcommand{\Htp}[1]{\mathbf{K}({#1})}

\newcommand{\Hplus}[1]{\mathbf{K}^+({#1})}
\newcommand{\Hpos}[1]{\mathbf{K}^{\ge 0}({#1})}
\newcommand{\Hneg}[1]{\mathbf{K}^{\le 0}({#1})}
\newcommand{\Hmock}[1]{\mathbf{K}_m({#1})}

\newcommand{\Der}[1]{\mathbf{D}({#1})}

\DeclareMathOperator{\colim}{colim}
\DeclareMathOperator{\Sum}{Sum}
\DeclareMathOperator{\Add}{Add}
\DeclareMathOperator{\Loc}{Loc}

\newcommand{\A}{\mathcal{A}}
\newcommand{\B}{\mathcal{B}}
\newcommand{\C}{\mathcal{C}}
\newcommand{\D}{\mathcal{D}}
\newcommand{\E}{\mathcal{E}}
\newcommand{\F}{\mathcal{F}}
\newcommand{\G}{\mathcal{G}}
\newcommand{\I}{\mathcal{I}}

\newcommand{\K}{\mathcal{K}}
\newcommand{\clL}{\mathcal{L}}
\newcommand{\M}{\mathcal{M}}
\newcommand{\clP}{\mathcal{P}}
\newcommand{\clS}{\mathcal{S}}
\newcommand{\T}{\mathcal{T}}
\newcommand{\U}{\mathcal{U}}
\newcommand{\V}{\mathcal{V}}
\newcommand{\W}{\mathcal{W}}
\newcommand{\X}{\mathcal{X}}
\newcommand{\Y}{\mathcal{Y}}

\newcommand{\stC}{\underline{\C}}
\newcommand{\stF}{\underline{\F}}
\newcommand{\stS}{\underline{\clS}}
\newcommand{\stT}{\underline{\T}}
\newcommand{\stX}{\underline{\X}}
\newcommand{\stY}{\underline{\Y}}

\newcommand{\bbX}{\mathbb{X}}
\newcommand{\bbZ}{\mathbb{Z}}
\newcommand{\PoneC}{{\mathbb{P}_\mathbb{C}^1}}
\newcommand{\OX}[1]{\mathcal{O}_{#1}}

\newcommand{\la}{\longrightarrow}
\newcommand{\infl}{\rightarrowtail}
\newcommand{\defl}{\twoheadrightarrow}
\newcommand{\laplus}{{\overset{+}\la}}

\newcommand{\xypushout}{\ar@{}[ul]|(.3)*\txt{$\bigstar$}}
\newcommand{\xypullback}{\ar@{}[dr]|(.3)*\txt{$\bigstar$}}
\newcommand{\xymoremargin}[1]{\xy *+{#1} \endxy}
\newdir{ >}{{}*!/-5pt/@{>}}     
\newcommand{\arinfl}{\ar@{ >->}}
\newcommand{\ardefl}{\ar@{->>}}

\newcommand{\smallpmatrix}[1]{
  \left(\begin{smallmatrix} #1 \end{smallmatrix}\right)
}

\usepackage{xcolor}

\begin{document}

\begin{abstract}
We show that Quillen's small object argument works for exact categories under very mild conditions. This has immediate applications to cotorsion pairs and their relation to the existence of certain triangulated adjoint functors and model structures. In particular, the interplay of different exact structures on the category of complexes of quasi-coherent sheaves leads to a streamlined and generalized version of recent results obtained by Estrada, Gillespie, Guil Asensio, Hovey, J\o{}rgensen, Neeman, Murfet, Prest, Trlifaj and possibly others.
\end{abstract}

\maketitle

\setcounter{tocdepth}{1}
\tableofcontents

\section*{Introduction}

Within the last few years, several authors presented very inventive results which at some point relied on some sort of approximation theory for exact categories. This dependency was, however, also the sticky part. Various original and elegant ideas had to carry a load of technicalities necessary to construct the approximations, a fact that can considerably hinder further development. Let us be more specific on what results we have in mind:
\begin{enumerate}
\item As explained in~\cite{Nee4}, Neeman~\cite{Nee3,Nee} and Murfet~\cite{Mur} gave a new interpretation of Grothendieck duality. At some point, they needed to construct a triangulated adjoint functor without having Brown representability (in the sense of~\cite[8.2.1]{Nee2}) at their disposal. Their workaround relied on the fact that certain classes of complexes of sheaves were precovering.

\item Following a program described in~\cite{H3}, Hovey~\cite{H1} and Gillespie~\cite{G3,G2,G1} constructed monoidal model structures on categories of complexes of sheaves. That is, the model structures are compatible with the tensor product, making, among other things, the standard properties of the derived tensor product obvious. The construction heavily relies on some cotorsion pairs being complete. See also~\cite{EGPT} for the most recent development.
\end{enumerate}

As one sees from comments like~\cite[\S8.7]{Nee2} or the paragraph before~\cite[Definition 7.8]{H3}, existence results for approximations and adjoint functors have always been viewed as something intriguing. This is despite the fact that we already have monographs like~\cite{GT} having approximations as their main topic. One reason may be that the results on approximations have not been put into a suitable context. For instance, \cite{GT} deals with modules while the papers mentioned before would preferably need the corresponding theory at least for Grothendieck categories.

In this situation, we felt that one should foster future progress by backing up the developments above by an appropriate approximation theory. Our setup is based on the intended applications---we study approximations and cotorsion pairs in exact categories which satisfy certain widespread left exactness and local smallness conditions.

Somewhat surprisingly, this approach brings ingredients which are new, or at least not well covered in the literature, also to the original setting of~\cite{GT}: approximation theory for module categories. In particular we prove that every deconstructible class of modules is precovering. Although the main tool we use for this, the so-called Quillen's small object argument, has been advocated by Hovey~\cite{H3} and actually used for constructing approximations by Rosick\'y~\cite{Ro}, its real power has not been exploited.

\smallskip

Let us briefly outline the structure of the paper. After recalling terminology and preliminary facts, we develop the technical core in Section~\ref{sec:small_obj}. What we present in Theorem~\ref{thm:orthogonality of morphisms - cotorsion pairs} and its corollary can be summarized as follows. Given a ``nice'' exact category $\C$ and a set $\clS$ of objects, then:
\begin{enumerate}
 \item The closure of $\clS$ under all transfinite extensions is a precovering class in $\C$.
 \item If either $\C$ has enough projectives or $\clS$ is a generating set, then the cotorsion pair $(\F,\T)$ generated by $\clS$ is complete and $\F$ is the smallest class which is closed under transfinite extensions and direct summands and contains $\clS$ and all projective objects.
\end{enumerate}

In Section~\ref{sec:frobenius} we establish the connection of cotorsion pairs to triangulated categories, triangulated adjoint functors, t-structures, weight structures and Bousfield localizations. Although the results presented there are hardly surprising, the section is important for having a solid background for applications.

Finally, in Section~\ref{sec:applications} we establish results directly related to intended applications to triangulated categories and model structures, and illustrate them on several examples. In particular, we give:

\begin{itemize}
\item in \S\ref{subsec:well generated categories} a handy criterion to recognize that some subcategories of the homotopy category of complexes are well generated. As an illustrative example, we look at the results of J\o{}rgensen~\cite{Jo} and Neeman~\cite{Nee3} for $\Htp{\Proj{R}}$;

\item in \S\ref{subsec:adjoints without Brown representability} criteria for existence of left and right adjoints to fully faithful triangulated functors without using Brown representability. As examples, we look among others at results by Neeman~\cite{Nee3,Nee} and Murfet~\cite{Mur};

\item in \S\ref{subsec:constructing monoidal model structures} a general condition which allows us to construct the derived category $\Der\G$ of a Grothendieck category $\G$ using ``$\F$-resolutions'' and ``$\C$-co\-res\-olutions'' of complexes. Here, $\F$ and $\C$ are suitable classes of objects of $\G$. This is closely related to the work of Gillespie~\cite{G3,G2,G1}, and Estrada, Guil, Prest and Trlifaj~\cite{EGPT}.
\end{itemize}

When finishing the work on our paper, we learned about an alternative result regarding the construction of approximations and adjoint functors due to Krause~\cite{Kr6}. We discuss briefly the relation between the two methods in Remark~\ref{rem:relation to El Bashir's result}.

\section{Preliminaries}
\label{sec:prelim}

\subsection{A little terminology from model categories}
\label{subsec:model categories}

We start with recalling some terminology from~\cite[Chapter 2]{H2} and~\cite[Chapter 10]{Hir} which is necessary for the technical core of the paper in Section~\ref{sec:small_obj}. We will also use some set-theoretic concepts, for which we refer to~\cite{Jech}.

Given a class $\I$ of morphisms in a category $\C$, a morphism $j:X\la Y$ is called a {\bf pushout of a map in $\mathcal{I}$} provided there exists a pushout
$$
\xymatrix{ A \ar[r]^{} \ar[d]_{i} & X \ar[d]^{j}
\\
B \ar[r]_{} & Y \xypushout }
$$
in $\C$ with $i\in\I$. We say that $\mathcal{I}$ is {\bf closed under pushouts} when a pushout of any morphism in $\I$ is
again in $\I$.

A direct system $[(X_\alpha)_{\alpha<\lambda},(i_{\alpha\beta})_{\alpha<\beta<\lambda})]$ indexed by an ordinal number $\lambda$ is called a {\bf $\lambda$-sequence} if, for each limit ordinal $\gamma<\lambda$, the colimit $\colim_{\alpha<\gamma}X_\alpha$ exists and  the colimit morphism $\colim_{\alpha<\gamma}X_\alpha \la X_\gamma$ is an
isomorphism. If a colimit of the whole direct system exists, we call the morphism $X_0\la \colim_{\alpha<\lambda}X_\alpha$ the {\bf composition} of the $\lambda$-sequence. If $\M$ is a class of morphisms in $\C$ and every morphism $i_{\alpha,\alpha+1}: X_\alpha \la X_{\alpha+1}$ for $\alpha+1<\lambda$ belongs to $\M$, we refer to the composition of the $\lambda$-sequence as a {\bf transfinite composition} of morphisms of $\M$.

A {\bf relative $\I$-cell complex} will stand for a transfinite composition of pushouts of maps from $\I$. If $\C$ has an initial object $0$, then an object $X$ in $\C$ is called an {\bf $\I$-cell complex} provided that the unique morphism $0 \la X$ is a relative $\I$-cell complex. The class of all $\I$-cell complexes is denoted by $\ICells$.

A morphism $p$ in $\C$ is called {\bf $\I$-injective} if it has the {\bf right lifting property} \wrt maps in $\I$: That is, given a commutative square
$$
\xymatrix{ A \ar[r]^{} \ar[d]_{i} & X \ar[d]^{p}
\\
B \ar[r]_{} & Y}
$$
in $\C$ with $i \in \I$, there exists $h: B \la X$ making both triangles commutative. The class of all $\I$-injective morphisms will be denoted by $\Iinj$.

\smallskip

In the sequel, we will need a few basic notions from~\cite[pp. 31--32]{Jech}. Given an ordinal $\lambda$, we say that an increasing sequence $(\lambda_\xi)_{\xi<\mu}$ of strictly smaller ordinals indexed by an ordinal $\mu$ is {\bf cofinal} in $\lambda$ if $\lambda = \sup_{\xi<\mu} \lambda_\xi$. The {\bf cofinality} of $\lambda$, denoted by $\cf \lambda$, is the smallest ordinal $\mu$ for which there is a cofinal sequence $(\lambda_\xi)_{\xi<\mu}$. We recall that $\cf\lambda$ is always a cardinal, see~\cite[Lemma 3.8]{Jech}.

A cardinal $\kappa$ is called is called {\bf regular} is $\kappa = \cf\kappa$, that is, $\kappa$ cannot be obtained as a sum of a collection of fewer than $\kappa$ cardinals, all of which are strictly smaller than $\kappa$. The countable cardinal $\aleph_0$ is regular (see~\cite[Example 10.1.12]{Hir}), and so is any infinite cardinal successor by~\cite[Lemma 5.8]{Jech} or~\cite[Proposition 10.1.14]{Hir}. However, there do exist cardinals which are not regular, for example $\aleph_\omega = \sup_{n<\omega} \aleph_n$.

\smallskip

In the spirit of~\cite[\S 2.1.1]{H2}, we now define what it means for an object $X$ of $\C$ to be small. Let $\kappa$ be an infinite regular cardinal and $\M$ a class of morphisms in $\C$. We say that $X$ is {\bf $\kappa$-small relative to $\M$} if for all ordinals $\lambda$ of cofinality greater than or equal to $\kappa$ and all $\lambda$-sequences
$$ Y_0 \la Y_1 \la Y_2 \la \dots \la Y_\alpha \la Y_{\alpha+1} \la \dots $$
\st the composition exists in $\C$ and $Y_\alpha \la Y_{\alpha+1}$ belongs to $\M$ for each $\alpha+1<\lambda$, the canonical morphism
$$
\colim_{\alpha <\lambda} \Hom_\C(X,Y_\alpha) \la \Hom_\C(X,\colim_{\alpha <\lambda}Y_\alpha)
$$
is an isomorphism. Finally, $X$ is said to be {\bf small relative to $\M$} if it is $\kappa$-small relative to $\M$ for some infinite regular cardinal $\kappa$.

Note that this smallness property is very natural and widespread. For instance, in an arbitrary accessible category $\C$ in the sense of Definition~\ref{def:accessible categories}, any object $X$ is small relative to the class $\M$ of all morphisms in $\C$.

\subsection{Basics on exact categories}
\label{subsec:exact categories}

The main topic of our paper is the approximation theory for exact categories $\C$. The concept of an exact category is originally due to Quillen~\cite{Q}, but the common reference for a simple axiomatic description is~\cite[Appendix A]{Kst} and an extensive treatment of the concept is also given in~\cite{Bu}. A few properties of exact categories for which we could not find a suitable reference are discussed in Appendix~\ref{sec:appendix exact catg} of this paper.

For our results, we will mostly consider exact categories meeting two additional axioms:

\begin{setup} \label{setup:nice exact category}
Let $\C$ be an exact category satisfying the following axioms:
\begin{enumerate}
\item[(Ax1)] Arbitrary transfinite compositions of inflations exist and are again inflations;
\item[(Ax2)] Every object of $\C$ is small relative to the class of all inflations.
\end{enumerate}
\end{setup}

In fact, a similar property restricted to compositions of countable chains of inflations was considered in~\cite[Appendix B]{Kst}. To convince the reader that these two conditions, although somewhat technical, are natural and very often met, we present two broad classes of examples to keep in mind:

\begin{exem} \label{expl:grothendieck category}
Let $\C$ be a Grothendieck category with the abelian exact structure. That is, we take for conflations all short exact sequences. Then it is well known that $\C$ satisfies both (Ax1) and (Ax2).
\end{exem}

\begin{exem} \label{expl:semisplit structure}
Let $\A$ be an additive category with arbitrary coproducts and define $\C = \Cpx\A$, the category of all chain complexes over $\A$:
$$
\cdots \la A^{-2} \overset{\partial^{-2}}\la
A^{-1} \overset{\partial^{-1}}\la
A^0 \overset{\partial^0}\la
A^1 \overset{\partial^1}\la A^2 \la \cdots
$$
The morphisms in $\C$ are usual chain complex morphisms and we equip $\C$ with the componentwise split exact structure. The fact that $\C$ satisfies (Ax1) of Setup~\ref{setup:nice exact category} follows from the easily verifiable fact that the same axiom holds for the additive category $\A$ itself, when taken with the split exact structure. If all objects of $\A$ are small \wrt the class of all split monomorphisms, such as if $\A$ is an accessible category in the sense of Definition~\ref{def:accessible categories} later on (see also~\cite[Chapter 2]{AR}), then $\C$ satisfies (Ax2) as well. This holds in particular if $\A = \Mod{R}$ for some ring $R$ or, more generally, if $\A$ is a Grothendieck category.

More generally, if $\A$ is an exact category satisfying (Ax1) and (Ax2), then the same is true for the category $\Cpx\A$ of chain complexes over $\A$, with conflations defined as the sequences of complexes which are conflations in $\A$ in each component.
\end{exem}

The following helpful lemma shows that exact categories satisfying (Ax1) always have arbitrary coproducts and these are exact:

\begin{lemma} \label{lem:existence of exact coproducts}
The following hold for any exact category $\C$ satisfying \emph{(Ax1)} of Setup~\ref{setup:nice exact category}:
\begin{enumerate}
    \item The category $\C$ has arbitrary coproducts.
    \item Coproducts of conflations are conflations.
\end{enumerate}
\end{lemma}

\begin{proof}
(1) Suppose we are given a family $(A_\alpha)_{\alpha\in K}$ of objects of $\C$. We can assume that $K=\lambda$ is an ordinal. Then the coproduct $\coprod_{\alpha<\lambda} A_\alpha$ can be constructed as the codomain of the composition of a $\lambda$-sequence. Namely, we inductively construct $[(X_\alpha)_{\alpha<\lambda},(i_{\alpha\beta})_{\alpha <\beta<\lambda}]$ \st for each $\alpha+1<\lambda$ we have $X_{\alpha+1} = X_\alpha \oplus A_\alpha$ and $i_{\alpha,\alpha+1}: X_\alpha \infl X_{\alpha+1}$ is the canonical split morphism. Here we use (Ax1) for limit steps.

(2) Suppose $(i_\alpha: A_\alpha\infl B_\alpha)_{\alpha\in K}$ is a family of inflations and assume again that $K=\lambda$ is an ordinal. Let us put $X_\alpha =(\coprod_{\beta<\alpha} B_\beta) \oplus (\coprod_{\beta\geq\alpha}A_\beta)$ for each $\alpha\le\lambda$. Note that we have $X_{\alpha+1}=(\coprod_{\beta<\alpha}B_\beta)\oplus B_\alpha \oplus (\coprod_{\beta\geq\alpha+1}A_\beta)$, so that we have an obvious inflation $j_{\alpha,\alpha+1}:X_\alpha\infl X_{\alpha+1}$. It is straightforward to inductively extend this data to a $\lambda$-sequence $[(X_\alpha)_{\alpha<\lambda},(j_{\alpha\beta})_{\alpha <\beta<\lambda})]$ and to check that $\coprod i_\alpha: \coprod_{\alpha<\lambda}A_\alpha \la \coprod_{\alpha<\lambda} B_\alpha$ is a transfinite composition of the inflations $j_{\alpha,\alpha+1}$, hence itself an inflation by (Ax1).
\end{proof}

\smallskip

Finally, we set up appropriate definitions for projective and injective objects and for generators. We call an object $I$ of an
exact category  $\C$ {\bf injective} if any inflation $i: I\infl Y$ splits. Equivalently, $\Ext^1_\C(Z,I) = 0$ for each $Z \in \C$. Here, given arbitrary objects $X$ and $Y$ in $\C$, we denote by $\Ext_\C^1(X,Y)$ the corresponding version of the {\bf Yoneda Ext}, that is, the abelian group of equivalence classes of conflations $0\rightarrow Y\longrightarrow Z\longrightarrow X\rightarrow 0$. We refer to~\cite[\S\S XII.4 and 5]{McL} for basic properties. We say that $\C$ has {\bf enough injectives} provided that for every object $X$, there is an inflation $X\infl I$ with $I$ injective. The terms {\bf projective object} and the existence of {\bf enough projectives} are defined dually.

A class $\G$ of objects of $\C$ is said to be {\bf generating} or a {\bf class of generators} of $\C$ if for
any object $M\in\C$ there is a deflation $\pi: G\defl M$, where $G$ is a (set-indexed) coproduct of objects from $\G$. An object $G \in \C$ is called a {\bf generator} if $\{G\}$ is generating.

\begin{rem} \label{rem:on projectives and generators, case I}
If $\C$ is a Grothendieck category with the abelian exact structure (Example~\ref{expl:grothendieck category}), it must have a generator by the very definition, and it is well known that $\C$ has enough injectives. On the other hand, there are several examples coming from algebraic geometry without enough projectives.
\end{rem}

\begin{rem} \label{rem:on projectives and generators, case II}
Assume now $\C = \Cpx\A$ is the category of complexes over $\A$ with the componentwise split exact structure as in Example~\ref{expl:semisplit structure}. Then $\C$ always has enough projectives and injectives by~\cite[\S I.3]{Ha}, but often no generating set. To be more specific, if $\A = \Mod{R}$ for a ring $R$, it is not hard to see that the existence of a generating set $\G$ of $\C$ forces the existence of a module $X \in \Mod{R}$ such that $\Mod{R}=\Add X$. In such a case $R$ is necessarily a right pure-semisimple ring (cf. \cite[Proposition 2.2]{St}). This is a very restrictive condition on $R$, for example $R = \bbZ$ does not satisfy it.
\end{rem}

\begin{rem}
By Proposition~\ref{prop:exact category with enough injectives}, an exact category $\C$ satisfies (Ax1) provided it has enough injectives, it has transfinite compositions of inflations, and it has the property that each section has a cokernel.
\end{rem}

\subsection{Approximations and cotorsion pairs}
\label{subsec:approx}

For homological algebra, it is crucial that certain classes provide for approximations, the classes of projective and injective modules giving the most well known examples. Let us introduce the necessary terminology here.

Given a class $\F$ of objects of a category $\C$, we say that a morphism $f: F \la X$ is an {\bf $\F$-precover} if $F \in \F$ and any other morphism $f': F' \la X$ with $F' \in \F$ factors through $f$. Sometimes, $\F$-precovers are also called right $\F$-approximations or weak $\F$-coreflections. The class $\F$ is called {\bf precovering} if each object $X$ of $\C$ admits an $\F$-precover $f: F \la X$. The notions of an {\bf $\F$-preenvelope} (also known as a left $\F$-approximation or a weak $\F$-reflection) and a {\bf preenveloping class} are defined dually.

In connection with approximations, it has proved useful to consider so called cotorsion pairs---see for instance~\cite{BEE,GT}. This concept, originating in a more than three decade old work of Salce~\cite{Sal}, generalizes in a straightforward way from module categories to an exact category $\C$. Given a class $\clS$ of objects of $\C$, we shall denote by $\clS^{\perp_1}$ the class of objects $\{Y \mid \Ext_\C^1(S,Y) = 0 \textrm{ for all } S \in \clS \}$. Dually, ${^{\perp_1}\clS}$ stands for the class $\{X \mid \Ext_\C^1(X,S) = 0 \textrm{ for all } S \in \clS \}$.

We call a pair $(\X,\Y)$ of full subcategories of $\C$ a {\bf cotorsion pair} provided that $\X^{\perp_1}=\Y$ and $\X={^{\perp_1}\Y}$. If $\clS$ is any class
of objects, then $(^{\perp_1}(\clS^{\perp_1}),\clS^{\perp_1})$ is always a cotorsion pair, called the {\bf cotorsion pair generated}%
\footnote{In the literature this is sometimes called the cotorsion pair cogenerated by $\clS$; then the cotorsion pair generated by $\clS$ is $(^{\perp_1}\clS,(^{\perp_1}\clS)^{\perp_1})$. Here, however, we use the terminology from~\cite{GT}.}%
{\bf by $\clS$}. A cotorsion pair $(\X,\Y)$ is called {\bf complete}
if for every object $M\in\C$ we have two conflations, usually called
{\bf approximation sequences}, of the form
$$
0\to Y\la X\la M\to 0
\qquad \textrm{and} \qquad
0\to M\la Y'\la X'\to 0
$$
with $X,X'\in\mathcal{X}$ and $Y,Y'\in\mathcal{Y}$. As one can readily check, the deflation $X \defl M$ is necessarily an $\X$-precover while the inflation $M \infl Y'$ is a $\Y$-preenvelope.

We will show in the following section that there is a general
construction for precovers, preenvelopes and complete cotorsion
pairs, which is important for applications.

\section{Exploiting the small object argument}
\label{sec:small_obj}

We start with Quillen's famous small object argument. We state it in a form which is based on Hovey's presentation in~\cite{H2}. Note that if $\C$ is an exact category as in Setup~\ref{setup:nice exact category} and $\M$ is the class of all inflations, then the assumptions below are satisfied for any subset $\I \subseteq \M$.

\begin{prop}[Quillen's small object argument] \label{prop:small object argument}
Let $\C$ be an arbitrary category and $\M$ be a class of morphisms satisfying the following
properties:

\begin{enumerate}
\item Arbitrary pushouts of morphisms in $\M$ exist and belong again to $\M$.
\item Arbitrary coproducts of morphisms in $\M$ exist and belong again to $\M$.
\item Arbitrary transfinite compositions of maps in $\M$ exist and belong to $\M$.
\end{enumerate}
Suppose that $\I \subseteq \M$ is a set of maps such that, for every $i:A\la B$ in $\I$, the domain $A$ is small with respect to relative $\mathcal{I}$-cell complexes.
Then every morphism $f:X\la Y$ in $\C$ admits a factorization $f=pj$, where $j$ is a relative $\I$-cell complex and $p$
is $\I$-injective.
\end{prop}

\begin{proof}
Exactly the same argument as for \cite[Theorem 2.1.14]{H2} or \cite[Proposition 10.5.16]{Hir} applies.
\end{proof}

Now, we can start with preparations to employ the small object argument. We state an auxiliary lemma first.

\begin{lemma} \label{lem:auxiliary lemma}
Let $\C$ be as in Setup~\ref{setup:nice exact category} and $\I$ be a set of inflations. The following assertions
hold:

\begin{enumerate}
\item The class of relative $\mathcal{I}$-cell complexes is closed under taking pushouts and coproducts.
\item If $p$ is a deflation in $\Iinj$, then the map $\Ker(p)\longrightarrow 0$ belongs to $\Iinj$.
\end{enumerate}
\end{lemma}

\begin{proof}
(1) This is very similar to the proof of \cite[Lemma 2.1.13]{H2}. Any relative $\I$-cell complex $i:A\la B$ is a composition of a $\lambda$-sequence $[(B_\alpha)_{\alpha<\lambda},(i_{\alpha\beta})_{\alpha <\beta <\lambda}]$ \st all $B_\alpha \infl B_{\alpha+1}$ are pushouts of morphisms from $\I$. If $u_0:A\la X$ is any morphism, we construct by transfinite induction pushouts
$$
\xymatrix{
B_\alpha \ar[d]_{u_\alpha} \ar[r]^{i_{\alpha,\alpha+1}} &
B_{\alpha+1}  \ar[d]^{u_{\alpha +1}}
\\
Y_\alpha \ar[r]_{j_{\alpha,\alpha+1}} &
Y_{\alpha+1} \xypushout
}
$$
Clearly $[(Y_\alpha)_{\alpha<\lambda},(j_{\alpha\beta})_{\alpha <\beta <\lambda}]$ is a $\lambda$-sequence whose composition is the pushout of $i$ along $u_0$.

In order to prove that the class of relative $\I$-cell complexes is closed under coproducts, suppose we are given such a family $(i_\alpha: A_\alpha\infl B_\alpha)_{\alpha<\lambda}$. Then we construct a $\lambda$-sequence $[(X_\alpha)_{\alpha<\lambda},(j_{\alpha\beta})_{\alpha <\beta<\lambda})]$ exactly as in the proof of Lemma \ref{lem:existence of exact coproducts}(2). It is easy to see that each $X_\alpha \la X_{\alpha+1}$ is a pushout of $i_\alpha: A_\alpha \la B_\alpha$ along the split inclusion $A_\alpha \infl X_\alpha$. It follows by the first part that $\coprod i_\alpha: \coprod A_\alpha \la \coprod B_\alpha$ is a transfinite composition of relative $\I$-cell complexes, hence $\coprod i_\alpha$ is a relative $\I$-cell complex itself by~\cite[2.1.12]{H2}.

(2). Let $p:X\twoheadrightarrow Y$ be a deflation in $\Iinj$. Then the conflation $0 \to \Ker(p) \longrightarrow X \longrightarrow Y \to 0$ gives rise to the pullback diagram:
$$
\xymatrix{
\Ker(p) \ar[r]^j \ar[d] \xypullback &
X \ar[d]^p
\\
0 \ar[r] &
Y
}
$$
We need to prove that if $i: A \la B$ is a morphism in $\I$ and $f: A \la \Ker(p)$ is any morphism, then there is a morphism $g: B \la \Ker(p)$ \st $gi = f$. To see that, note that the fact that $p \in \Iinj$ implies the existence of a morphism $h: B \la X$ \st $hi = jf$ and $ph = 0$. By the pullback property, we get a unique morphism $g: B \la \Ker(p)$ \st $jg = h$, so that $jgi = jf$. Hence $gi = f$ since $j$ is a monomorphism. The latter is an instance of a more general fact, namely that $\Iinj$ is closed under taking pullbacks.
\end{proof}

Next, we make definitions inspired by~\cite{G1}, which will be very useful in our study of cotorsion pairs in exact categories.

\begin{defi} \label{def:homological set of inflations}
Let $\C$ be an arbitrary exact category and let $\I$ be a set of inflations. We denote by $\Coker(\I)$ the class of objects isomorphic to $\Coker(i)$,
where $i$ runs over elements of $\I$. We shall say that

\begin{enumerate}
\item $\I$ is {\bf homological} if the following two conditions are equivalent for any object $T\in\C$:

\begin{enumerate}
\item $\Ext_\C^1(S,T)=0$ for all $S\in\Coker(\I)$;
\item The map $T\la 0$ belongs to $\mathcal{I}$-inj (that is, the map $\Hom_\C(i,T): \Hom_\C(B,T) \la \Hom_\C(A,T)$ is surjective,
for all $i: A\la B$ in $\I$).
\end{enumerate}

\smallskip

\item $\mathcal{I}$ is {\bf strongly homological} if given any inflation $j:A\rightarrowtail B$ whose cokernel $S$ belongs to $\Coker(\mathcal{I})$, there is a morphism $i:A'\rightarrowtail B'$ in $\mathcal{I}$ giving rise to a commutative diagram with conflations in rows:
$$
\xymatrix{
0 \ar[r] &
A' \ar[d] \ar[r]^i &
B' \ar[d] \ar[r] &
S \ar@{=}[d] \ar[r] &
0
\\
0 \ar[r] &
A \ar[r]^j &
B \ar[r] &
S \ar[r] &
0
}
$$
\end{enumerate}
\end{defi}

\begin{rem} \label{rem:str.homological implies small Ext}
Note that given a set of objects $\clS$, the existence of a strongly homological set of inflations $\I$ \st $\Coker(\I) = \clS$ is equivalent to the following condition: For each object $S \in \clS$ there is a family of objects $(A_j)_{j \in J}$ together with an epimorphism of functors $\coprod_{j \in J} \Hom_\C(A_j,-) \la \Ext_\C^1(S,-)$. In particular, $\Ext_\C^1(S,Y)$ is a set rather than a proper class for each $S \in \clS$ and $Y \in \C$. Although having ``set-sized'' extension spaces is a very natural property for an exact category, it does not come for free, see~\cite[Exercise 1, p. 131]{Fr}.
\end{rem}

The terminology of Definition~\ref{def:homological set of inflations} is justified by the following simple lemma:

\begin{lemma} \label{lem:homological sets}
Given any exact category $\C$, a strongly homological set of inflations is always homological.
\end{lemma}

\begin{proof}
The implication (a) $\implies$ (b) in Definition~\ref{def:homological set of inflations} is always satisfied, so we only have to prove (b) $\implies$ (a). Let $\I$ be a strongly homological set of inflations and let $T$ be an object such that $T\longrightarrow 0$ belongs to $\Iinj$. Let us fix $S\in \Coker(\I)$ and let $\varepsilon\in \Ext_\C^1(S,T)$ be an extension represented by the conflation $0 \to T \overset{j}\la X \la S \to 0.$ By the assumption on $\I$, we get a commutative diagram with conflations in rows and $i\in\I$:
$$
\xymatrix{
&
0 \ar[r] &
A' \ar[d]_f \ar[r]^i &
B' \ar[d] \ar[r] &
S \ar@{=}[d] \ar[r] &
0 &
\\
\varepsilon: &
0 \ar[r] &
T \ar[r]^j &
X \ar[r] &
S \ar[r] &
0 &
}
$$
The left hand square is a pushout by~\cite[Proposition 2.12]{Bu}. Since $T \la 0$ is $\I$-injective, $f$ factors through $i$ and $\varepsilon$ splits. Hence $\Ext_\C^1(S,T) = 0$ for each $S \in \Coker(\I)$.
\end{proof}

In order to obtain reasonable properties, we now add one more axiom to Setup~\ref{setup:nice exact category}:

\begin{defi} \label{def:efficient exact category}
An exact category $\C$ will be called {\bf efficient} if it satisfies the following conditions:
\begin{enumerate}
\item[(Ax1)] Arbitrary transfinite compositions of inflations exist and are again inflations;
\item[(Ax2)] Every object of $\C$ is small relative to the class of all inflations;
\item[(Ax3)] For each set of objects $\clS$, there is a strongly homological set of inflations $\I$ \st  $\Coker(\I)$ is the class of objects isomorphic to objects of $\clS$.
\end{enumerate}
Here, (Ax1) and (Ax2) are precisely the conditions of Setup~\ref{setup:nice exact category}.
\end{defi}

A crucial observation is that many natural exact categories are indeed efficient
(see Examples~\ref{expl:efficient exact categories} below).

\begin{prop} \label{prop:examples of strongly homological sets of maps}
Let $\C$ be an exact category \st

\begin{enumerate}
\item $\C$ has enough projectives, or
\item $\C$ has a set of generators, arbitrary coproducts, and every section in $\C$ has a cokernel.
\end{enumerate}
Then, for each set of objects $\clS$, there is a strongly homological set of inflations $\I$ such that $\Coker(\I)$ is the class of objects isomorphic to objects of $\clS$. In particular, if $\C$ satisfies \emph{(Ax1)} and \emph{(Ax2)} of Definition~\ref{def:efficient exact category}, then it is an efficient exact category.
\end{prop}


\begin{proof}
To prove the statement when $\C$ has enough projectives, we only need to fix a conflation
$$ 0 \to K_S \overset{i_S}\la P_S \la S \to 0, $$
with $P_S$ projective, for each $S \in \clS$. A standard argument for module categories, also valid here, shows that $\Ext^1_\C(S,T) \cong \Coker \Hom_\C(i_S,T)$. Hence the set of inflations $\I= \{ i_S \mid S\in\clS \}$ is strongly homological.

The proof for case (2) is inspired by~\cite{G1}, while a similar argument is also used in the proof of~\cite[Lemma 4.3]{EGPT}. If there is a set of generators of $\C$, we can, by taking their coproduct, obtain a single generator $G \in \C$. Let $S\in\clS$ and suppose we have a conflation of the form
$$ 0 \to A \overset{u}\la B \overset{q}\la S \to 0. $$

We first claim that the canonical morphism $\pi: G^{(\Hom_\C(G,B))}\la B$ is a deflation. To see that, fix a deflation $\pi': G^{(I)} \defl B$, for some set $I$, and  denote by $\pi'_i: G\la B$ its $i$-th component, for each $i\in I$. Take now a set $I' \subseteq I$ of representatives of equivalence classes for the equivalence relation on $I$ given by $i\sim j$ if $\pi'_i = \pi'_j$. It is obvious that $(\pi'_i)_{i\in I'}$ is a weakly terminal subfamily of $(\pi'_i)_{i\in I}$ (see Definition~\ref{def:weakly terminal subfamily}). Lemma~\ref{lem:weakly terminal deflation} tells us that the induced morphism $\pi'': G^{(I')} \la B$ is also a deflation.
Viewing $I'$ as a subset of $\Hom_\C(G,B)$ in the obvious way, the morphism $\pi: G^{(\Hom_\C(G,B))} \la B$ can be expressed as
$$ \begin{pmatrix}\pi'' & \rho\end{pmatrix}: G^{(I')} \oplus G^{(\Hom_\C(G,B) \setminus I')} \la B, $$
and the claim follows by applying Lemma~\ref{lem:adding a deflation}.

Further, let us define a subset $J \subseteq \Hom_\C(G,S)$ as $J = \{ q\circ f \mid f \in \Hom_\C(G,B) \}$. As the composition $q\pi: G^{(\Hom_\C(G,B))}\la S$ is a deflation, one gets, invoking Lemma~\ref{lem:weakly terminal deflation} once again, that the natural morphism $\rho: G^{(J)} \la S$ is a deflation. That is, we have a commutative diagram
$$
\xymatrix{
0 \ar[r] &
K_{S,J} \ar[r]^{i_{S,J}} \ar[d] &
G^{(J)} \ar[r]^{\rho} \ar[d]_h &
S \ar[r] \ar@{=}[d] &
0
\\
0 \ar[r] &
A \ar[r]^{u} &
B \ar[r]^{q} &
S \ar[r] &
0
}
$$
with conflations in rows, where the morphisms $h: G^{(J)} \la B$ is constructed so that the component $h_j: G \la B$ corresponding to $j \in J \subseteq \Hom_\C(G,S)$ satisfies the equality $j = q h_j$.

It follows from the construction that the the set $\I = \{ i_{S,J}: K_{S,J} \infl G^{(J)} \}$, where $S$ runs over all elements of $\clS$ and $J$ runs over all subsets of $\Hom_\C(G,S)$ \st $\rho: G^{(J)} \la S$ is a deflation, is strongly homological and, up to isomorphism, $\clS = \Coker(\I)$.
\end{proof}

\begin{exems} \label{expl:efficient exact categories}
Using the last proposition, we get the following classes of examples of efficient exact categories:

\begin{enumerate}
\item A Grothendieck category with the abelian structure is efficient; see also Example~\ref{expl:grothendieck category}.

\item An accessible additive category $\A$ with arbitrary coproducts (see Definition~\ref{def:accessible categories}), considered with the split exact structure, is efficient. Similarly, the category $\Cpx\A$ of chain complexes  over such $\A$ with the componentwise split exact structure is efficient; see also Example~\ref{expl:semisplit structure}.

\item A Frobenius exact category (see Section~\ref{sec:frobenius}) is efficient \iff it satisfies the conditions of Setup~\ref{setup:nice exact category}. Using Proposition~\ref{prop:exact category with enough injectives}, it suffices to check for (Ax1) that all sections have cokernels and that all transfinite compositions of inflations exist.
\end{enumerate}
\end{exems}

As for module categories~\cite{GT}, a crucial tool for studying classes of objects in an exact category defined by vanishing of the $\Ext$ functor are so called filtrations. Let us give a definition:

\begin{defi} \label{def:filtrations}
Let $\C$ be an exact category and $\clS$ be a class of objects of $\C$. Then an object $X$ of $\C$ is called {\bf $\clS$-filtered} if the morphism $0\la X$ is the  composition of a $\lambda$-sequence $[(X_\alpha)_{\alpha<\lambda},(i_{\alpha\beta})_{\alpha<\beta<\lambda}]$ \st all $X_\alpha \la X_{\alpha+1}$ are inflations
with a cokernel in $\clS$. The $\lambda$-sequence is then called an {\bf $\clS$-filtration} of $X$, and the class of all $\clS$-filtered objects is denoted by $\Filt\clS$.
\end{defi}

While filtrations are in fact just transfinite extensions and in module categories the concept is very natural, at the level of generality we are dealing with one has to be careful. For instance, the following lemma for module categories with the abelian exact structure would be almost obvious.

\begin{lemma} \label{lem:refining filtrations}
Let $\C$ be as in Setup~\ref{setup:nice exact category}. Suppose we are given a conflation
$$ \varepsilon: \quad 0 \to A \overset{j}\la B \la F \to 0 $$
and a $\lambda$-sequence $[(F_\alpha)_{\alpha<\lambda},(i_{\alpha\beta})_{\alpha <\beta<\lambda}]$ \st $F_0 = 0$, $\colim_{\alpha<\lambda} F_\alpha = F$, and all the morphisms $i_{\alpha,\alpha+1}$ are inflations. Then there is a $\lambda$-sequence
$$ [(B_\alpha)_{\alpha<\lambda},(m_{\alpha\beta})_{\alpha <\beta<\lambda}] $$
whose composition is precisely $j: A \la B$ and \st all the morphisms $m_{\alpha,\alpha+1}$ are inflations and $\Coker m_{\alpha,\alpha+1} \cong \Coker i_{\alpha,\alpha+1}$ for all $\alpha<\lambda$.
\end{lemma}

\begin{proof}
It is fairly easy to see what $[(B_\alpha)_{\alpha<\lambda},(m_{\alpha\beta})_{\alpha <\beta<\lambda}]$ should be. Namely, we can construct a well ordered direct system of conflations
$$ \varepsilon_\alpha: \quad 0\rightarrow A\stackrel{j_\alpha}\longrightarrow B_\alpha\longrightarrow F_\alpha\rightarrow 0 $$
as pullbacks of the original conflation $\varepsilon$ along the colimit morphisms $i_\alpha: F_\alpha \infl F$. Clearly, $j_0$ is an isomorphism, so we can take $B_0 = A$ and $j_0 = 1_A$. This way, we get a direct system $[(B_\alpha)_{\alpha<\lambda},(m_{\alpha\beta})_{\alpha <\beta<\lambda}]$. Using the axioms of exact categories and~\cite[Proposition 2.15]{Bu}, we immediately see that $m_{\alpha,\alpha+1}$ are all inflations and $\Coker m_{\alpha,\alpha+1} \cong \Coker i_{\alpha,\alpha+1}$.

The non-trivial part here is that this direct system is a $\lambda$-sequence and that $j: A \la B$ is its composition. What we will in fact prove is that for any limit ordinal $\gamma<\lambda$, the colimit morphism $\colim_{\alpha<\gamma}\varepsilon_\alpha \la \varepsilon_\gamma$ is an isomorphism, an so is the morphism $\colim_{\alpha<\lambda}\varepsilon_\alpha \la \varepsilon$. We prove the claim by induction on $\gamma\le\lambda$, where we put $\varepsilon_\lambda = \varepsilon$ by definition. To this end, first note that for each such $\gamma$ the colimit diagram
$$
\colim_{\alpha<\gamma} \varepsilon_\alpha: \quad
0\to A\la \colim_{\alpha<\gamma}B_\alpha\la \colim_{\alpha<\gamma}F_\alpha\to 0
$$
is a conflation since the morphism $A\la \colim_{\alpha<\gamma}B_\alpha$ is none other than the composition of the $\gamma$-sequence $[(B_\alpha)_{\alpha<\gamma},(m_{\alpha\beta})_{\alpha <\beta<\gamma}]$. Here we use the inductive hypothesis and the fact that direct limits commute with cokernels. Then the morphism $\colim_{\alpha<\gamma}\varepsilon_\alpha \la \varepsilon_\gamma$ gives rise to a diagram with conflations in rows:
$$
\xymatrix{
0 \ar[r] &
A \ar[r] \ar@{=}[d] &
\colim_{\alpha<\gamma} B_\alpha \ar[r] \ar[d] &
\colim_{\alpha<\gamma} F_\alpha \ar[r] \ar[d]^\cong &
0\phantom{,}
\\
0 \ar[r] &
A \ar[r]^{j_\alpha} &
B_\alpha \ar[r] &
F_\alpha \ar[r] &
0,
}
$$
which makes clear that $\colim_{\alpha<\gamma}B_\alpha \la B_\gamma$ is an isomorphism (use for example~\cite[Proposition 2.12]{Bu}).
\end{proof}

As an immediate consequence, using what is called interpolation of sequences in~\cite[Definition 10.2.11]{Hir}, we have the following:

\begin{cor} \label{cor:filtrations closed under transf ext}
Let $\C$ be as in Setup~\ref{setup:nice exact category}, $\clS$ be a class of objects of $\C$, and put $\F = \Filt\clS$. Then any $\F$-filtered object of $\C$ belongs to $\F$.
\end{cor}

\begin{proof}
Let $[(X_\alpha)_{\alpha<\lambda},(i_{\alpha\beta})_{\alpha <\beta<\lambda}]$ be an $\F$-filtration of $X$, that is, a $\lambda$-sequence \st $X_0 = 0$, $X = \colim_{\alpha<\lambda} X_\alpha$ and $\Coker(i_{\alpha,\alpha+1}) \in \Filt\clS$ for all $\alpha+1 < \lambda$. For each $\alpha+1 < \lambda$, let us fix an $\clS$-filtration $[(F^\alpha_\gamma)_{\gamma<\mu_\alpha},(i^\alpha_{\gamma\delta})_{\gamma <\delta<\mu_\alpha}]$ of $\Coker(i_{\alpha,\alpha+1})$. Hence, Lemma~\ref{lem:refining filtrations} provides us with a $\mu_\alpha$-sequence $[(B^\alpha_\gamma)_{\gamma<\mu_\alpha},(m^\alpha_{\gamma\delta})_{\gamma<\delta<\mu_\alpha}]$ whose composition is $i_{\alpha,\alpha+1}: X_\alpha \la X_{\alpha+1}$. A standard argument allows us to obtain an $\clS$-filtration of $X$ by interpolating the $\F$-filtration by the $\mu_\alpha$-sequences above. We refer to~\cite[Proposition 10.2.8]{Hir} for details.
\end{proof}

More importantly though, we prove a result which can be interpreted as the fact that the left hand class of a cotorsion pair (cf. Section~\ref{subsec:approx}) is closed under transfinite extensions. For module categories, this is known as the Eklof Lemma, see~\cite[Lemma 3.1.2]{GT}.

\begin{prop} \label{prop:eklof lemma}
Let $\C$ be an exact category as in Setup~\ref{setup:nice exact category}. Suppose we have a $\lambda$-sequence $[(X_\alpha)_{\alpha<\lambda},(i_{\alpha\beta})_{\alpha<\beta<\lambda}]$ with composition $0 \la X$, and \st all the morphisms $i_{\alpha,\alpha+1}$ are inflations. If $Y \in \C$ is \st $\Ext_\C^1(\Coker i_{\alpha,\alpha+1},Y) = 0$ for each $\alpha<\lambda$, then also $\Ext_\C^1(X,Y) = 0$.
\end{prop}

\begin{proof}
Denote $S_\alpha = \Coker i_{\alpha,\alpha+1}$ for each $\alpha<\lambda$. We need to prove that whenever
$$
0\to A\overset{j}\la B\la X\to 0
$$
is a conflation and $f: A \la Y$ is a morphism, then $f$ factors through $j$. Indeed, if this is the case, then any conflation $0\to Y\la E\la X\to 0$ splits.

Here, however, we apply Lemma~\ref{lem:refining filtrations} to get a $\lambda$-sequence $[(B_\alpha)_{\alpha<\lambda},(m_{\alpha\beta})_{\alpha <\beta<\lambda}]$ with composition equal to $j: A \la B$, and \st for all $\alpha<\lambda$ we have a conflation
$$
0 \la B_\alpha \xrightarrow{m_{\alpha,\alpha+1}} B_{\alpha+1} \xrightarrow{\phantom{m_{\alpha,\alpha+1}}} S_\alpha \la 0.
$$

Now, if we are given any morphism $f: A \la Y$, we can inductively extend it to a morphism $f_\alpha: B_\alpha \la Y$. Namely, we take $f_0 = f$, $f_\alpha$ always extends to $f_{\alpha+1}$ since $\Ext_\C^1(S_\alpha,Y)=0$, and at limit steps we just take $f_\alpha=\colim_{\alpha<\beta} f_\beta$.
\end{proof}

Now we are ready to connect the preceding results together in the main result of the section. Being somewhat technical due to its generality, we supplement the theorem with Corollary~\ref{cor:consequences of Theorem}, which is easier to use in our suggested applications.

\begin{theor} \label{thm:orthogonality of morphisms - cotorsion pairs}
Let $\C$ be an exact category as in Setup~\ref{setup:nice exact category}, that is, \st transfinite compositions of inflations exist and are inflations, and \st every object of $\C$ is small relative to the class of all  inflations. Let $\I$ be a homological set of inflations, and put $\clS=\Coker(\I)$ and $\T=\clS^{\perp_1}$. The
following assertions hold:
\begin{enumerate}
\item The class $\ICells$ is precovering in $\C$;

\item $\ICells \subseteq \Filt\clS$. If $\I$ is strongly homological, then $\ICells = \Filt\clS$;

\item $\T = (\ICells)^{\perp_1} = (\Filt\clS)^{\perp_1}$;

\item The class $\T$ is a preenveloping class in $\mathcal{C}$. In fact, for each object $M\in\C$, there is a conflation
$$
0 \to M \la T \la X \to 0,
$$
with $T\in\T$ and $X\in\ICells$;

\item If the class $^{\perp_1}\T$ is generating, then for each object $M\in\C$ there is a conflation
$$
0\to T' \la X' \la M \to 0,
$$
with $T'\in\T$ and $X'\in {^{\perp_1}\T}$. In particular, $(^{\perp_1}\T,\T)$ is a complete
cotorsion pair in $\C$;

\item If $\Filt\clS$ is a generating class of $\C$, then ${^{\perp_1}\mathcal{T}}$ consists precisely of direct summands of $\clS$-filtered objects.
\end{enumerate}
\end{theor}

\begin{proof}
(1). Our version of the small object argument from Proposition~\ref{prop:small object argument} applies here. If $M\in\C$ is any object, then we decompose the map $0\la M$ as
$$ 0\overset{i}\la X \overset{p} \la M,$$
where $i$ is a relative $\I$-cell complex and $p\in\Iinj$. Therefore, $X\in\ICells$. We claim that $p$ is actually an $\ICells$-precover. Indeed, if $Z\in\ICells$ then the map $0\la Z$ is a relative $\I$-cell complex. In particular, all $\I$-injective morphisms must have the right lifting property with respect to $0\la Z$. Applying the fact to $p \in \Iinj$, this precisely means that every map $f:Z\la M$ factors through $p$. Thus assertion (1) holds.

(2). The inclusion $\ICells \subseteq \Filt\clS$ follows easily from Definition~\ref{def:filtrations}, since a pushout of a morphism from $\I$ is always an inflation with cokernel isomorphic to an object from $\clS$.

On the other hand, if $\I$ is strongly homological, then any inflation $A\infl B$ with a cokernel in $\clS$ is a relative $\I$-cell complex since it is a pushout of a map in $\I$ by Definition~\ref{def:homological set of inflations} and~\cite[Proposition 2.12]{Bu}. Since the class of relative $\I$-cell complexes is closed under taking transfinite compositions (cf.\ \cite[Lemma 2.1.12]{H2}), the inclusion $\Filt\clS\subseteq\ICells$ follows.

(3). In view of (2), it suffices to prove that $\Ext_\C^1(X,T) = 0$ for any $\clS$-filtered object $X$ and $T \in \T$. But this immediately follows from Proposition~\ref{prop:eklof lemma} and the fact that $\Ext_\C^1(S,T) = 0$ for each $S \in \clS$.

(4). First note that since $\I$ is homological, an object $T$ belongs to $\T=\clS^{\perp_1}$ \iff the morphism $T\la 0$ belongs $\Iinj$. The proof of $\T$ being preenveloping is then entirely dual to the proof of (1). That is, given $M \in \C$, we use the small object argument to find a factorization $M \overset{i}\la T \overset{p}\la 0$ of $M \la 0$, where $i$ is a relative $\I$-cell complex and $p \in \Iinj$, and readily show that $T \in \T$ and $i$ is a $\T$-preenvelope.
Moreover, $i$ is necessarily an inflation and $X=\Coker(i)\in\ICells$. The latter holds since $0\longrightarrow X$ is a pushout of $i$, hence a relative $\I$-cell complex by Lemma~\ref{lem:auxiliary lemma}.
%
%

(5). We first claim that ${^{\perp_1}\T}$ is closed under taking coproducts. Indeed, if we express the coproduct of a family $(A_\alpha)_{\alpha\in K}$ of objects of $\C$ as an $\clS$-filtration exactly as in the proof of Lemma~\ref{lem:existence of exact coproducts}, it suffices to apply Proposition~\ref{prop:eklof lemma} to see that $\coprod_{\alpha \in K} A_\alpha \in {^{\perp_1}\T}$. This proves the claim.

This allows us for each object $M$ to fix a deflation  $q:X''\la M$ with $X''\in {^{\perp_1}\T}$. Now, we put $K=\Ker(q)$ and apply assertion (4) to get a conflation $ 0\to K\la T\la X\la 0$, where $T\in\T$ and $X\in\ICells$. Following an idea of Salce from~\cite{Sal}, we form a pushout diagram with conflations in rows and columns:
$$
\xymatrix{
0 \ar[r] &
\xymoremargin{K} \ar[r] \arinfl[d] &
\xymoremargin{T} \ar[r] \arinfl[d] &
X \ar[r] \ar@{=}[d] &
0
\\
0 \ar[r] &
X'' \ar[r] \ardefl[d]_{q} &
X'  \ar[r] \ardefl[d] \xypushout &
X   \ar[r] &
0
\\
&
M \ar@{=}[r] &
M
}
$$
It follows that $X'\in{^{\perp_1}\T}$ since both $X$ and $X''$ have
this property. The central column of the diagram is then our desired
conflation.

(6). Assume $M \in {^{\perp_1}\T}$ and take a conflation
$$ 0\to K\la X''\overset{q}\la M\to 0 $$
with $X'' \in \Filt\clS$. This is possible since $\Filt\clS$ is closed under taking coproducts (cf. Corollary~\ref{cor:filtrations closed under transf ext}) and we assume that it is a generating class. Now, we apply assertion (4) to get a conflation
$$ 0\to K\la T\la X\to 0, $$
where $T\in\T$ and $X\in\ICells \subseteq \Filt\clS$. By constructing the same pushout diagram as above, we get a conflation
$$ 0\to T\la X'\la M\to 0 $$
with $T \in \T$ and $X' \in \Filt\clS$. Since $\Ext_\C^1(M,T) = 0$ by assumption, the latter conflation splits and $M$ is a summand of $X' \in \Filt\clS$.
\end{proof}

The precise relation between homological and strongly homological sets of morphisms is not fully understood yet. Consider the following example.

\begin{exem} \label{expl:homological vs strongly homological}
Let $\C$ be the category of abelian groups with the abelian exact structure and put $\I = \{ i: \mathbb{Z} \infl \mathbb{Q} \}$ and $\clS = \{ \mathbb{Q}/\mathbb{Z} \}$. Since $T = \mathbb{Q} \oplus \mathbb{Q}/\mathbb{Z}$ is an infinitely generated tilting abelian group in the sense of~\cite[Definition 5.1.1]{GT} (see also~\cite[Example 5.1.3]{GT}), the class $\clS^{\perp_1}$ is precisely the class of divisible groups by~\cite[Corollary 5.1.10(b)]{GT}. If now $Y$ is any abelian group such that $\Hom_\C(i,Y)$ is surjective, then every element $y \in Y$ of the form $y = f(x)$, for some morphism $f: \mathbb{Q} \la Y$ and $x = i(1)$. Therefore the trace of $\mathbb{Q}$ in $Y$ is the whole $Y$ and, hence, $Y$ is divisible. In particular, $\I$ is homological.

On the other hand, $\I$ cannot be strongly homological. One can see that by taking any free resolution $0 \to \mathbb{Z}^{(\aleph_0)} \overset{j}\la \mathbb{Z}^{(\aleph_0)} \la \mathbb{Q}/\mathbb{Z} \to 0$ of $\mathbb{Q}/\mathbb{Z}$. Then any commutative square with $i$ at the top and $j$ at the bottom has zero columns.

Nevertheless, we have $\ICells = \Filt\clS = \Add\{\mathbb{Q}/\mathbb{Z}\}$. This is because $\mathbb{Q}/\mathbb{Z}$ is injective in $\C$, so every $\clS$-filtration necessarily splits at each step.
\end{exem}

As promised before, we give an easier to apply corollary. Note that the conditions of $\C$ having a generator and $\C$ having enough projectives are independent; recall Remarks~\ref{rem:on projectives and generators, case I} and~\ref{rem:on projectives and generators, case II}.

\begin{cor} \label{cor:consequences of Theorem}
Assume $\C$ is an efficient exact category (see Definition~\ref{def:efficient exact category} and Examples~\ref{expl:efficient exact categories}), let $\clS$ be a set of objects of $\C$ and put $\T=\clS^{\perp_1}$. The following hold:
\begin{enumerate}
\item $\Filt\clS$ is a precovering class and $\T$ is a preenveloping class.

\item If $\clS$ is a generating set  of $\C$,
then $({^{\perp_1}\T}, \T)$ is a complete cotorsion pair in $\C$ and ${^{\perp_1}\T}$ consists precisely of summands of $\clS$-filtered objects.

\item If $\C$ has enough projectives, then $(^{\perp_1}\T,\T)$ is a complete cotorsion pair and an object $M$ is in ${^{\perp_1}\T}$ \iff it is a direct summand of an object
$E$ appearing in a conflation of the form
$$ 0 \to P \la E \la X \to 0, $$
with $P$ projective and $X \in \Filt\clS$.
\end{enumerate}
\end{cor}

\begin{proof}
By definition, there is a strongly homological set of inflations $\I$ \st (up to isomorphism) $\Coker(\I)=\clS$. Then Theorem~\ref{thm:orthogonality of morphisms - cotorsion pairs}(1) and~(2) tell us that $\Filt\clS=\ICells$ is a precovering class. $\T$ is preenveloping by Theorem~\ref{thm:orthogonality of morphisms - cotorsion pairs}(4). This proves the first part. The second part is a direct consequence of Theorem~\ref{thm:orthogonality of morphisms - cotorsion pairs}(5)
and (6).

The third part is a variant of the proof of Theorem~\ref{thm:orthogonality of morphisms - cotorsion pairs}(6). Assume $M \in {^{\perp_1}\T}$. We have a conflation $0 \to K \la P \la M \to 0$ with $P$ projective. Theorem~\ref{thm:orthogonality of morphisms - cotorsion pairs}(4) yields a conflation $0 \to K \la T \la X \to 0$ with $T \in \T$ and $X \in \Filt\clS$. If we form the pushout diagram
$$
\xymatrix{ 0 \ar[r] & \xymoremargin{K} \ar[r] \arinfl[d] &
\xymoremargin{T} \ar[r] \arinfl[d] & X \ar[r] \ar@{=}[d] & 0
\\
0 \ar[r] & P \ar[r] \ardefl[d] & E \ar[r] \ardefl[d] \xypushout & X
\ar[r] & 0
\\
& M \ar@{=}[r] & M, }
$$
the middle column splits since $\Ext_\C^1(M,T)=0$, while the middle
row gives the desired conflation.
\end{proof}

\begin{rem} \label{rem:relation to El Bashir's result}
There is a well-known result by El Bashir~\cite{ElB,BEE} providing precovering classes under very general conditions. It says that a class $\F$ in a Grothendieck category $\G$ is precovering if it is closed under coproducts, direct limits, and there is a subset $\clS \subseteq \F$ \st each $X \in \F$ is a direct limit of objects of $\clS$.

It turns out that this is a consequence of our
Corollary~\ref{cor:consequences of Theorem}(1). Namely, one can show that there is a set $\clS' \subseteq \F$ \st $\F = \Filt\clS'$ in $\G$ considered with the exact structure formed by all $\kappa$-pure exact sequences for a suitable regular cardinal $\kappa$. This essentially follows from the theory for accessible categories~\cite[\S2]{AR} (cf.\ also Definition~\ref{def:accessible categories}) and~\cite[Lemma 3.1]{ElB}, since $\F$ as above can be checked to be accessible using standard arguments.

Several versions of El Bashir's result have appeared in the literature since then, to our best knowledge the most general of them being the one obtained in~\cite{Kr6} by Krause. We do not know whether Krause's result follows from our considerations, but our method has one advantage. As in~\cite{EGPT}, we do \emph{not} need our precovering classes to be closed under direct limits---an assumption essential in~\cite{ElB,Kr6}. Such an assumption would not allow us to prove that the class of locally projective quasi-coherent sheaves (see Example~\ref{expl:mock homotopy categories of projectives}) is precovering.
\end{rem}

\section{Cotorsion pairs in Frobenius exact categories}
\label{sec:frobenius}

All through this section, $\C$ is a {\bf Frobenius exact category}, that is, it has enough projectives and enough injectives in the sense of Section~\ref{subsec:exact categories}, and the injective and projective objects coincide.
We shall denote by $\stC$ the stable category of $\C$ modulo the projective objects. This is a triangulated category (see~\cite[\S I.2]{Ha}) with the suspension functor $?[1]=\Omega^{-1}$, the Heller functor, which assigns to each object $X$ the ``cosyzygy'' object $\Omega^{-1}X$ given by a conflation
$$ 0\to X\la I\la \Omega^{-1}X\to 0 $$
with $I$ injective. Note that $\Omega^{-1}X$ is uniquely determined (up to a unique isomorphism) in $\stC$, but not in $\C$. Such triangulated categories are called {\bf algebraic}, see~\cite{K,Kr}. The group of homomorphisms between objects $X$ and $Y$ in $\stC$ will be in a customary way denoted by $\stHom_\C(X,Y)$.

If $\X \subseteq \C$ is a full subcategory, we shall denote by $\stX$ the essential image of the composition $\X \infl \C \overset{\mathbf{p}}\la \stC$,
where $\mathbf{p}$ is the projection functor. Since $\C$ and $\stC$ have the same objects,
we can also view $\stX$ as a full subcategory of $\C$. The following lemma describes
 the relation between $\X$ and $\stX$ in $\C$:

\begin{lemma} \label{lem:projective closure}
An object $Y$ is in $\stX$ \iff there exist $X \in \X$ and $P,Q$
projective in $\C$ \st $X \oplus P \cong Y \oplus Q$ in $\C$.
\end{lemma}

\begin{proof}
This is easy and well-known.
\end{proof}

Our goal here is to describe the interplay between cotorsion pairs in $\C$ and the corresponding concepts in $\stC$. To this end, we need the following general definition.

\begin{defi} \label{def:Hom-orthogonal pairs}
Let $\D$ be an additive category and $\clS$ be a class of objects. We denote $\clS^\perp = \{Y\in \D \mid \Hom_\D(S,Y)=0 \textrm{ for all } S\in\clS\}$ and dually ${^\perp \clS} = \{X\in \D \mid \Hom_\D(X,S)=0 \textrm{ for all } S\in\clS\}$. Then we call a pair $(\X,\Y)$ of full subcategories a {\bf Hom-orthogonal pair} if $\X^\perp = \Y$ and $\X = {^\perp \Y}$. Given $\X_0 \subseteq \X$ \st $\Y=\X_0^\perp$, we say that the Hom-orthogonal pair is {\bf generated by $\X_0$}.
\end{defi}

We see in the following proposition that Hom-orthogonal pairs in $\stC$ relate in a somewhat technical but very straightforward way to cotorsion pairs in $\C$. Here, by $\Omega X$ we mean a kernel of a deflation $P \defl X$ with $P$ projective, and $\Omega\F$ stands for a class $\{\Omega X \mid X \in \F\}$.

\begin{prop} \label{prop:Hom-orthogonal pairs versus cotorsion pairs}
Let $(\F,\T)$ be a pair of full subcategories of a Frobenius exact
category $\C$. The following assertions are equivalent:

\begin{enumerate}
\item $(\stF,\stT)$ is a cotorsion pair in $\C$;
\item $(\underline{\Omega\F},\stT)$ is a Hom-orthogonal pair in $\stC$.
\end{enumerate}
That is, the assignment $(\F,\T) \rightsquigarrow
(\underline{\Omega\F},\stT)$ establishes a bijective correspondence
between cotorsion pairs in $\C$ and Hom-orthogonal pairs in $\stC$.
Moreover, if  we put $(\X,\Y) = (\underline{\Omega\F},\stT)$, then
the following assertions are equivalent:

\begin{enumerate}
\item[(a)] $(\F,\T)$ is a complete cotorsion pair.
\item[(b)] For every object $M$ of $\C$, there is a conflation $0 \to M \la T \la F \to 0$ in $\C$ with $T\in\T$ and $F\in\F$.
\item[(c)] For every object $M$ of $\stC$, there is a triangle $X \la M \la Y \laplus$ in $\stC$ with $X\in\X$ and $Y\in\Y$.
\end{enumerate}
\end{prop}

\begin{proof}
Since $\C$ is a Frobenius exact category, we have a natural isomorphism
$$ \stHom_\C(\Omega X,Y) \cong \Ext_\C^1(X,Y) $$
for all $X,Y \in \C$. The equivalence $(1)\Longleftrightarrow (2)$ follows easily. For the bijective correspondence, note that if $(\F,\T)$ is a cotorsion pair in $\C$, then $\F = \stF$ and $\T =
\stT$.

Regarding the equivalence between (a), (b) and (c), the implication (a) $\implies$ (b) is obvious. To show (b) $\implies$ (c), note that the conflation $0\to M\la T\la F\to 0$ gives a triangle $M\la T\la F\laplus$, where $T \in \stT = \Y$ and $F \in \stF = \X[1]$. Hence the desired triangle is obtained just by shifting.

Finally for (c) $\implies$ (a), we take the triangles for given $M$
and its syzygy $\Omega M$:
$$
X \la M \la Y \laplus
\qquad \textrm{and} \qquad
X'\la \Omega M \la Y' \laplus
$$
By shifting the triangles correspondingly, we obtain $M \la Y \la
X[1] \laplus$ and $Y' \la X'[1] \la (\Omega M)[1] \laplus$. Note
that $M \cong (\Omega M)[1]$ in $\stC$, and further we have $X[1],
X'[1] \in \X[1] = \F$ and $Y, Y' \in \Y = \T$. Using the explicit
construction of triangles in $\stC$, we obtain conflations
$$ 0 \to M \la Y\oplus Q \la \Omega^{-1}X \oplus P \to 0, $$
$$ 0 \to Y'\oplus Q' \la \Omega^{-1}X'\oplus P' \la M \to 0, $$
with $P,P',Q,Q'$ projective. Using Lemma~\ref{lem:projective
closure} and the equalities $\stF=\F$ and $\stT=\T$, we get that the
cotorsion pair $(\F,\T)$ is complete.
\end{proof}

\begin{warn} \label{warn:warning for t-structures}
The approximation triangles from Proposition~\ref{prop:Hom-orthogonal pairs versus cotorsion pairs}(c) are \emph{not} necessarily unique up to isomorphism in general, see Proposition~\ref{prop:uniqueness of the triangle} later in the section.
\end{warn}

A direct consequence of Proposition~\ref{prop:Hom-orthogonal pairs
versus cotorsion pairs} is the following.

\begin{cor} \label{cor:Distinguished triangle of a Hom-orthogonal pair}
Let $\C$ be an efficient Frobenius exact category (eg.\ $\C = \Cpx\A$ as in Example~\ref{expl:semisplit structure}). Assume further that $\X_0 \subseteq \C$ is a set of objects and $(\X,\Y)$ is the Hom-orthogonal pair in $\stC$ generated by $\X_0$. Then each object $M \in \stC$ appears in a triangle
$$ X\la M\la Y\laplus, $$
where $X\in\X$ and $Y\in\Y$. Moreover $X$ can be chosen, as an object of $\C$, to be $\X_0$-filtered.
\end{cor}

\begin{proof}
Obviously, the cotorsion pair $(\F,\T)$ in $\C$ corresponding to the Hom-orthog\-onal pair $(\X[-1],\Y[-1])$ in $\stC$ via
Proposition~\ref{prop:Hom-orthogonal pairs versus cotorsion pairs} is the cotorsion pair generated by the set $\X_0$.
Theorem~\ref{thm:orthogonality of morphisms - cotorsion pairs}(4)  gives us a conflation
$$ 0 \to \Omega M \la T \la X\to 0, $$
where $X\in \Filt\X_0$ and $T\in\T$. After shifting it, we get a triangle $X \la (\Omega M)[1] \la T[1] \laplus$ in $\stC$. Here again, $(\Omega M)[1] \cong M$ in $\stC$, $X \in \F = \X$ by Proposition~\ref{prop:eklof lemma}, and $T[1] \in \T[1] = \Y$.
\end{proof}

\subsection{Cohereditary cotorsion pairs, t-structures and localizations}
\label{subsec:t-struct}

Now we will build up a dictionary between the language of cotorsion pairs and well known concepts related to triangulated categories. Let us recall definitions from~\cite[\S1.3]{BBD} and~\cite[\S6]{KS} (we also refer to~\cite{HRS,BR} and~\cite[Chapter 9]{Nee2} for more information on the notions):

\begin{defi} \label{def:t-struct}
Let $\D$ be a triangulated category with suspension $?[1]$. A full subcategory $\X$ is called {\bf suspended} if it is closed under extensions, direct summands, and
$\X[1]\subseteq\X$. A pair $(\X,\Y)$ of full subcategories is called a {\bf t-structure} if
\begin{enumerate}
    \item $(\X[1],\Y)$ is a Hom-orthogonal pair in $\D$,
    \item $\X[1] \subseteq \X$, and
    \item each $M \in \D$ lies in a triangle of the form
    $$ X[1]\longrightarrow M\longrightarrow Y\stackrel{+}{\longrightarrow}, $$
  with $X\in\X$ and $Y\in\Y$.
\end{enumerate}
Note that assuming (1), condition (2) is equivalent to saying that $\X$ is suspended. A t-structure $(\X,\Y)$ is called a {\bf Bousfield localizing pair} if $\X$ is a triangulated subcategory of $\T$, or equivalently if condition (2) above is replaced by:
\begin{enumerate}
    \item[(2')] $\X = \X[1]$.
\end{enumerate}
\end{defi}


\begin{rem}
The term Bousfield localizing pair is based on the terminology used by Neeman~\cite{Nee2} and Krause~\cite{Kr3}. Such pairs are also called {\bf semi-orthogonal decompositions} in the literature, following the terminology by Bondal and Orlov~\cite{BO}.
\end{rem}

\begin{rem} \label{rem:functorial triangles}
If $(\X,\Y)$ is a t-structure, then the triangle from Definition~\ref{def:t-struct}(3) is unique up to a unique isomorphism, hence functorial. Note that the factorization of morphisms in Proposition~\ref{prop:small object argument} can be made functorial (see~\cite[Theorem 2.1.14]{H2}) and, thus, so can the triangle from Corollary~\ref{cor:Distinguished triangle of a Hom-orthogonal pair}. However, this functoriality is not canonical, for it depends on choices made in the construction, and has not proved to be very useful so far. What really seems to matter is the uniqueness, which provides us naturally with adjoints to the inclusions $\X \infl \D$ and $\Y \infl \D$.
\end{rem}

Now we establish  a bijective correspondence between t-structures and certain cotorsion pairs. For this purpose, we call a cotorsion pair $(\F,\T)$ {\bf hereditary} if $\Omega \F \subseteq \F$ and {\bf cohereditary} if $\Omega^{-1} \F \subseteq \F$.

\begin{prop} \label{prop:uniqueness of the triangle}
Let $\C$ be a Frobenius exact category  and denote by $\clP$ the full subcategory of all projective objects. Further, let $(\F,\T)$ be a cotorsion pair in $\C$ and $(\X,\Y) = (\underline{\Omega\F},\stT)$ the associated Hom-orthogonal pair in
$\stC$ as in Proposition~\ref{prop:Hom-orthogonal pairs versus cotorsion pairs}. Then the following assertions are equivalent:

\begin{enumerate}
\item $(\F,\T)$ is a complete cotorsion pair and $\F \cap \T = \clP$;
\item $(\F,\T)$ is a complete cohereditary cotorsion pair;
\item For each $M \in \stC$, there is a triangle $X \la M \la Y \laplus$ in $\stC$ with $X \in \X$ and $Y \in \Y$, which is unique up to isomorphism;
\item The pair $(\X[-1],\Y)$ is a t-structure in $\stC$
\end{enumerate}
In particular, the assignment $(\F,\T) \rightsquigarrow (\X[-1],\Y) = (\underline{\Omega^2 \F},\stT)$ gives a bijective correspondence between complete cohereditary cotorsion pairs in $\C$ and
t-structures in $\stC$. It restricts to a bijective correspondence $(\F,\T) \rightsquigarrow (\X[-1],\Y) = (\stF,\stT)$ between
complete hereditary and cohereditary cotorsion pairs in $\C$ and Bousfield localizing pairs in $\stC$.
\end{prop}

\begin{proof}
$(1)\Longrightarrow (2)$. Assume $M \in \F$. We have to prove that $\Omega^{-1}M \in \F$ for some (or equivalently any) cosyzygy of $M$. Let $0 \to M \longrightarrow T \longrightarrow F\to 0$ be a conflation with $T \in \T$ and $F \in \F$. Then $T \in \F \cap \T = \clP$, so $F$ is a cosyzygy of $M$.

$(2)\Longrightarrow (4)$. From condition (2) we get that the Hom-orthogonal pair $(\X,\Y)$ in $\stC$ satisfies $\X[1] \subseteq \X$ and that every object $M$ of $\stC$ is the central term of a triangle
$$ X \la M \la Y \laplus, $$
with $X \in \X$ and $Y \in \Y$, see Proposition~\ref{prop:Hom-orthogonal pairs versus cotorsion pairs}. Then $(\X[-1],\Y)$ is a t-structure in $\stC$.

$(4)\Longrightarrow (3)$. This is a well-known basic property of t-structures, see Remark~\ref{rem:functorial triangles}.

$(3)\Longrightarrow (1)$. By Proposition \ref{prop:Hom-orthogonal pairs versus cotorsion pairs}, we know that $(\F,\mathcal{T})$ is a complete cotorsion pair. If now $Z\in\F\cap\mathcal{T}$ and $0 \to \Omega Z \longrightarrow P \longrightarrow Z \to 0$ is a conflation with $P$ projective, we get the following triangle in $\underline{\mathcal{C}}$:
$$ \Omega Z \la 0 \la Z \laplus. $$
Note that we have $\Omega Z \in \underline{\Omega\F} = \X$ and $Z \in \stT = \Y$. But there is another triangle with the middle term equal to zero, namely the one whose all terms are zero objects. The uniqueness property implies that $Z \cong 0$ in
$\underline{\mathcal{C}}$. Hence, $Z\in\mathcal{P}$ and so $\mathcal{S} \cap \mathcal{T} = \mathcal{P}$.

Finally, the fact that $(\F,\T) \rightsquigarrow (\X[-1],\Y)$ gives a bijective correspondence between the classes in the statement follows immediately from the facts above and and the equalities $\F = \stF$ and $\T = \stT$ for any cotorsion pair $(\F,\T)$ in $\C$.
\end{proof}

\begin{rem}
There are several results available in the literature which are formally similar to Proposition~\ref{prop:uniqueness of the triangle}, but they do not imply it.
See for instance \cite[Corollary V.3.8]{BR}, \cite[Theorem 7.12]{Bel}, \cite[Proposition 0.4]{KS}, \cite[Corollary 6.3]{KS} or~\cite[Proposition A.4]{KS}.
\end{rem}

By combining Proposition~\ref{prop:uniqueness of the triangle} with Corollary~\ref{cor:Distinguished triangle of a Hom-orthogonal pair}, one sees that each set $\X_0$ of objects of an efficient Frobenius exact category $\C$ gives rise to a t-structure via the cohereditary cotorsion pair generated by $\X_0$. We apply this observation in \S\ref{subsec:well generated categories}.

There are, however, interesting t-structures coming from cotorsion pairs which are \emph{not} generated by a set. The Bousfield localizing pair $(\Htp{\Flat{R}}, \Htp{\Flat{R}}^\perp)$ in $\stC = \Htp{\Mod{R}}$ studied in~\cite{Nee} may serve as an example. In fact, $\Htp{\Flat{R}}$ is generated by a set \iff $R$ is a right perfect ring by~\cite[Proposition 4.4 and Theorem 5.2]{St}.  This phenomenon will be discussed more in detail in~\S\ref{subsec:adjoints without Brown representability}. A key tool to overcome the problem of not having a generating set is the following result by Neeman~\cite{Nee}, a special version of which can be found already in~\cite[Proposition 1.3]{KV} by Keller and Vossieck:

\begin{prop} \label{prop:Neeman's}
Let $\D$ be a triangulated category with splitting idempotents and
$\X \subseteq \D$ be a precovering suspended full subcategory. Then
the inclusion functor $\X \infl \D$ has a right adjoint. In
particular, $(\X,\Y)$ is a t-structure, where $\Y = \X^\perp [1]$.
\end{prop}

\begin{proof}
To get the right adjoint, the proof of~\cite[Proposition 1.4]{Nee} can be copied mutatis mutandis. The fact that $\X$ is the left hand class of a t-structure follows in a standard way; see~\cite[\S1]{KV}.
\end{proof}

We now have a useful consequence. Note that a typical case when the conditions below are met is when $\C = \Cpx\G$ for a Grothendieck category $\G$, $\mathbf{E}$ is the class of all componentwise split short exact sequences in $\Cpx\G$ and $\mathbf{E}'$ is the class of all short exact sequences in $\Cpx\G$. We also note that in some situations, a more direct way to compute the right adjunction to $\stF \infl \stC$ has been obtained in~\cite{EBIJR}.

\begin{cor} \label{cor:moving between exact structures}
Let $\C$ be an additive category on which we consider two exact
structures, $\C_\mathbf{E}$ and $\C_{\mathbf{E}'}$, given by classes
of conflations $\mathbf{E} \subseteq \mathbf{E}'$. Suppose that the
following two conditions hold:
\begin{enumerate}
\item $\mathcal{C}_{\mathbf{E}'}$ is an efficient exact category (see Definition~\ref{def:efficient exact category}),
\item $\mathcal{C}_\mathbf{E}$ is a Frobenius exact category.
\end{enumerate}
If $\clS$ is a set of objects and $\F$ is the class of
$\clS$-filtered objects in $\C_{\mathbf{E'}}$, then
$(\underline{\F},\underline{\F}^{\perp}[1])$ is a t-structure
(resp.\ a Bousfield localizing pair) in $\stC_\mathbf{E}$ whenever
$\underline{\Omega_\mathbf{E}^{-1}(\F)} \subseteq \underline{\F}$
(resp. $\underline{\Omega_\mathbf{E}^{-1}(\F)} = \underline{\F}$).
\end{cor}

\begin{proof}
We simply check the assumptions of Proposition \ref{prop:Neeman's} for $\D=\stC_\mathbf{E}$ and $\X=\stF$. Since $\C_{\mathbf{E}'}$ is efficient, the underlying category $\C$ has arbitrary coproducts by Lemma~\ref{lem:existence of exact coproducts}. As a consequence, the triangulated category $\stC_\mathbf{E}$ has arbitrary coproducts, and consequently splitting idempotents by~\cite[I.6.8]{Nee2}. Since $\F$ is closed under coproducts and extensions in $\C_{\mathbf{E}'}$
by Corollary~\ref{cor:filtrations closed under transf ext} and $\mathbf{E} \subseteq \mathbf{E}'$, it follows that $\stF$ is closed under extensions and coproducts in $\stC_\mathbf{E}$. Note that $\stF$ is also closed under summands in $\stC_\mathbf{E}$, again by~\cite[I.6.8]{Nee2}. Now, the conditions $\underline{\Omega_\mathbf{E}^{-1}(\F)} \subseteq \underline{\F}$ (resp.\ $\underline{\Omega_\mathbf{E}^{-1}(\F)} = \stF$) precisely say that $\stF$ is a suspended (resp.\ triangulated) subcategory of $\stC_\mathbf{E}$. Finally, by applying Corollary~\ref{cor:consequences of Theorem} to the exact category $\C_{\mathbf{E}'}$, we see that $\F$ is a precovering class in $\C$, and so is $\stF$  in $\stC_\mathbf{E}$.
\end{proof}


\subsection{Hereditary cotorsion pairs and weight-structures}
\label{subsec:w-struct}

Cotorsion pairs are also closely related to weight structures in the sense of~\cite{B2,B,Pauk,MSSS}. We recall \cite[Definition 1.1.1]{B2}:

\begin{defi} \label{def:weight structure}
A {\bf weight structure} on a triangulated category $\D$ is a pair $(\X,\Y)$ of full subcategories satisfying the following properties:

\begin{enumerate}
\item $(\Y,\X[1])$ is a $\Hom$-orthogonal pair,
\item $\X[1] \subseteq \X$,
\item each object $M \in \D$ appears in a triangle of the form
$$ M \la X \la Y \laplus, $$
with $X \in \X$ and $Y \in \Y$.
\end{enumerate}
\end{defi}

\begin{rem} \label{rem:weight structure definition}
Since the formal definition of a weight structure is (somewhat deceptively) very similar to that of a t-structure, only interchanging the roles of $\X$ and $\Y$, the same notion has later appeared in~\cite{Pauk,MSSS} under the name {\bf co-t-structure}. The definition in papers~\cite{B2,B,Pauk,MSSS} is in fact slightly different, but equivalent by~\cite[Proposition 2.1]{Pauk}. In particular, in any weight structure $(\X,\Y)$ the class $\X$ is a suspended subcategory of $\D$.
\end{rem}

\begin{exem}
Let $\A$ be an additive category and $\D = \Htp\A$, the homotopy category of complexes over $\A$. Denote by $\Hneg\A$ and $\Hpos\A$ the full subcategories of complexes supported, up to homotopy isomorphism, only in non-positive and non-negative degrees, respectively. Then $(\Hneg\A,\Hpos\A)$ is the prototype of a weight structure on $\Htp\A$, see~\cite{B2,B}.
\end{exem}

Recall that we call a cotorsion pair $(\F,\T)$ in a Frobenius exact category hereditary if $\Omega\F \subseteq \F$ or, equivalently $\Omega^{-1} \T \subseteq \T$. The following is now an analogue of Proposition \ref{prop:uniqueness of the triangle}.

\begin{prop} \label{prop:weight structures versus cotorsion pairs}
Let $\C$ be a Frobenius exact category. Then the following are
equivalent for a pair $(\F,\T)$ of full subcategories of $\C$:
\begin{enumerate}
    \item $(\stF,\stT)$ is a complete hereditary cotorsion pair in $\C$;
    \item $(\stT,\stF)$ is a weight structure in $\stC$.
\end{enumerate}
In particular, the assignment $(\F,\T) \rightsquigarrow (\stT,\stF)$ gives a bijective correspondence between complete hereditary cotorsion pairs in $\C$ and weight structures on $\stC$.
\end{prop}

\begin{proof}
By Proposition \ref{prop:Hom-orthogonal pairs versus cotorsion pairs}, the assignment $(\F,\T) \rightsquigarrow
(\stF,\underline{\Omega^{-1}\T})$ gives a bijection between cotorsion pairs in $\C$ and Hom-orthogonal pairs in $\stC$, and
$(\F,\T)$ is complete \iff each $M \in \stC$ admits a triangle $M \la X \la Y \laplus$ with $X \in \stT$ and $Y \in \stF$. In such a case, it is easy to check using the equality $\stT[1] = \underline{\Omega^{-1}\T}$, that $(\stT,\stF)$ is a weight structure \iff $(\F,\T)$ is complete and hereditary.
\end{proof}

\section{Applications}
\label{sec:applications}

Now we shortly describe how to apply the relations between cotorsion pairs, exact categories and triangulated categories. This is related to recent work by several authors. For all applications, there is a common important concept which we will need:

\begin{defi} \label{def:deconstructible}
Let $\C$ be an exact category. We say that a full subcategory $\F$ is {\bf deconstructible} if there is a set $\clS$ of objects of $\C$ \st
$\F = \Filt\clS$ in $\C$ in the sense of Definition~\ref{def:filtrations}.
\end{defi}

That is, a class is deconstructible if we can get all its objects from a set by forming transfinite extensions. For deconstructible classes in Grothendieck categories, we refer to~\cite{St2} for more information. Now we can proceed to the applications.

\subsection{Structure of certain algebraic triangulated categories}
\label{subsec:well generated categories}

In this part, we give a handy tool to prove that certain algebraic triangulated categories are well generated. As easy consequences, we recover some results by J\o{}rgensen~\cite{Jo} and Neeman~\cite[\S5]{Nee3}. Let us give the main statement, which we explain, prove and apply in the following paragraphs.

\begin{theor} \label{thm:accessible frobenius exact categories}
Let $\C$ be an accessible Frobenius exact category with arbitrary coproducts. Then
\begin{enumerate}
\item $\stC$ is a locally well generated triangulated category.
\end{enumerate}
If, in addition, transfinite compositions of inflations exist in $\C$, then
\begin{enumerate}
\item[(2)] $\C$ is efficient (i.e. satisfies the axioms of Definition~\ref{def:efficient exact category}).
\item[(3)] For any set $\clS$ of objects of $\C$ \st $\stS = \stS[1]$, the equality $\Loc\stS = \underline{\Filt\clS}$ holds. That is, the smallest localizing subcategory of $\stC$ containing $\clS$ consists, up to homotopy isomorphism, precisely of the
$\clS$-filtered objects in $\C$.

\item[(4)] A localizing subcategory $\clL \subseteq \stC$ is well generated \iff $\clL = \stF$ for some deconstructible subcategory $\F \subseteq \C$. In that case, $(\clL ,\clL^\perp)$ is a Bousfield localizing pair in $\stC$.
\end{enumerate}
\end{theor}

Let us start with explaining the terminology. First we recall a few concepts, for whose basic properties we refer to~\cite{MaPa} or~\cite[Chapter 2]{AR}:

\begin{defi} \label{def:accessible categories}
Let $\D$ be a category and $\kappa$ be a regular cardinal. A direct limit in $\D$ is called {\bf $\kappa$-direct} if the indexing set $I$ of the direct system is $\kappa$-directed, that is, each subset of $I$ of cardinality less than $\kappa$ has an upper bound in $I$. An object $M \in \D$ is called {\bf $\kappa$-presentable} if the functor $\Hom_\D(M,-)$ preserves all $\kappa$-direct limits which exist in $\D$. The category $\D$ is called {\bf $\kappa$-accessible} if
\begin{enumerate}
 \item $\D$ has all $\kappa$-direct limits, and
 \item there is a set $\D_0$ of $\kappa$-presentable objects of $\D$ \st each $M \in \D$ is a $\kappa$-direct limit of objects from $\D_0$.
\end{enumerate}
Finally, $\D$ is called {\bf accessible} if it is $\kappa$-accessible for some regular cardinal $\kappa$.
\end{defi}

The general shape of accessible categories is presented in~\cite[Theorem 2.26]{AR}. For our purposes, the following observations is useful:

\begin{lemma} \label{lem:Grothendieck categories are accessible}
Any module category $\Mod{R}$ and, more generally, any Grothendieck category is accessible.
\end{lemma}

\begin{proof}
This is well known. If $\C$ is a Grothendieck category, it is $\kappa$-accessible for any infinite regular cardinal $\kappa$ \st $\kappa > \lvert \End_\C(G) \rvert$ for some generator $G$.
\end{proof}

In applications we wish to work with complexes of modules or sheaves, so it is crucial to know that accessibility is preserved when passing to categories of complexes.

\begin{prop} \label{prop: accessibility passes to complexes}
Let $\A$ be an accessible additive category. Then the category $\Cpx\A$ of all chain complexes over $\A$ is accessible.
\end{prop}

\begin{proof}
Using the obvious version of~\cite[Theorem 2.39]{AR} for additive categories and functors, one sees that for any small preadditive category $\U$, the category $\A^\U$ of all additive functors $F: \U \la \A$ is accessible. What remains to show is that $\Cpx\A$ is equivalent to $\A^\U$ for suitable $\U$. We define such $\U$ explicitly and leave the details about the equivalence $\A^\U \la \Cpx\A$ for the reader. The objects of $\U$ are integers and
$$
\Hom_\U(i,j) =
\left\{ \begin{array}{ll}
\mathbb{Z} & \textrm{for $j=i$ or $j=i+1$,} \\
0 & \textrm{otherwise.} \\
\end{array} \right.
$$
In particular,  if $\partial^i: i \la i+1$ is the generator of the free group $\Hom_\U(i,i+1)$, one has $\partial^{i+1}\circ\partial^i=0$ for all $i\in\mathbb{Z}$. Clearly, an object $X = (X^i,d^i)$ of $\Cpx\A$ is identified with the functor $X: \U \la \A$ taking $i\mapsto X^i$ and $\partial^i \mapsto d^i$.
\end{proof}

Let us turn our attention to well generated categories. The original definition was introduced by Neeman~\cite{Nee2} for triangulated categories, but as pointed out by Ji\v{r}\'\i{} Rosick\'y, it makes a good sense for any additive category with coproducts. The definition here differs slightly from Neeman's original definition, but it is equivalent for any triangulated category by~\cite[Theorem A]{Kr2} and~\cite[Lemmas 4 and 5]{Kr2}.

\begin{defi} \label{def:well generated}
Let $\D$ be an additive category with arbitrary coproducts and $\kappa$ be a regular cardinal. Then $\D$ is called {\bf $\kappa$-well generated} by a set of objects $\D_0 \subseteq \D$ provided that $\D_0$ satisfies the following conditions:
\begin{enumerate}
\item If $M \in \D$ is non-zero, there is a non-zero morphism $X \la M$ for some $X \in \D_0$;
\item each object $X \in \D_0$ is $\kappa$-small \wrt the class of all split monomorphisms in the sense of Section \ref{subsec:model categories};
\item For any morphism in $\D$ of the form $f: X \la \coprod_{i \in I} M_i$ with $X \in \D_0$, there exists a family of morphisms $f_i: X_i \la M_i$ \st $X_i \in \D_0$ for each $i \in I$ and $f$ factorizes as
$$
X \xrightarrow{\phantom{\coprod f_i}}
\coprod_{i \in I} X_i \xrightarrow{\coprod f_i}
\coprod_{i \in I} M_i.
$$
\end{enumerate}
The category $\D$ is called {\bf well generated} if it is
$\kappa$-well generated by a set of objects,  for some regular
cardinal $\kappa$.
\end{defi}

For algebraic triangulated categories, there is a weaker notion introduced in~\cite{St}, which is satisfied by many triangulated categories which are not well generated. Let us recall this together with some related concepts from~\cite{Nee2}.

\begin{defi} \label{def:locally well generated}
Let $\D$ be a triangulated category with arbitrary coproducts. A full triangulated subcategory $\clL \subseteq \D$ is called a {\bf localizing subcategory} if it is closed under forming coproducts in $\D$. If $\clS$ is a class of objects of $\D$, we denote the smallest localizing subcategory of $\D$ containing $\clS$ by $\Loc\clS$.

The triangulated category $\D$ is called {\bf locally well generated} if $\Loc\clS$ is a well generated category for any set $\clS$ of objects of $\D$.
\end{defi}

Now we have all necessary terminology and can prove the theorem.

\begin{proof}[Proof of Theorem~\ref{thm:accessible frobenius exact categories}]
(1) We learn from~\cite[Corollary 2.14]{AR} that there are arbitrary large regular cardinals $\kappa$ \st $\C$ is $\kappa$-accessible. Let us denote the class of all cardinals with this property by $\K$, and
let $\C_\kappa$ be the full subcategory of $\C$ formed by all $\kappa$-presentable objects. It follows from Definition~\ref{def:accessible categories} and~\cite[Proposition 1.16]{AR} that $\C = \bigcup_{\kappa \in \K} \C_\kappa$. Hence we
have
$$ \stC = \bigcup_{\kappa \in \K} \underline{\C_\kappa} = \bigcup_{\kappa \in \K} \clL_\kappa, $$
where we put $\clL_\kappa = \Loc \underline{\C_\kappa}$ in $\stC$.

We claim that $\clL_\kappa$ is $\kappa$-well generated for each $\kappa \in \K$. To this end, we can take a representative set $\D_0$ of $\C_\kappa$, which is possible by~\cite[Remark 2.15]{AR}. Then we prove that $\clL_\kappa$ is $\kappa$-well generated by $\D_0$ in the sense of Definition~\ref{def:well generated}. Indeed, each $X \in \D_0$ satisfies Definition~\ref{def:well generated}(2) because it is $\kappa$-presentable in $\C$. To prove condition (3) of Definition~\ref{def:well generated}, recall the well-known fact that if we have an arbitrary family $(M_i)_{i \in I}$ of objects of $\C$ and express each $M_i$ in $\C$ as a $\kappa$-direct limit $M_i = \varinjlim_{j \in J_i} X_{i,j}$ of $\kappa$-presentable objects, then we have the $\kappa$-direct limit
$$ \coprod_{i \in I} M_i = \varinjlim_{(j_i) \in \prod J_i} \big( \coprod_{i \in I} X_{i,j_i} \big). $$
Since each $X \in \D_0$ is $\kappa$-presentable, any morphism $X \la \coprod_{i \in I} M_i$ factors through $\coprod_{i \in I} X_{i,j_i}$ for some $(j_i) \in \prod_{i \in I} J_i$, as desired. Finally, condition (1) of Definition~\ref{def:well generated} easily follows from the fact that $\clL_\kappa = \Loc \underline{\C_\kappa} = \Loc \D_0$, see for example~\cite[Lemma 3.3]{St}. This proves the claim.

To finish the proof that $\stC$ is locally well generated, we proceed as in~\cite[Theorem 3.5]{St}.
Namely, if $\clS$ is a set of objects of $\stC$, then $\clS \subseteq \underline{\C_\kappa}$ for some $\kappa \in \K$.
As shown in~\cite[Corollary 7.2.2]{Kr3}, $\Loc\clS$ is then a $\kappa$-well generated triangulated subcategory of $\clL_\kappa$. Hence $\stC$ is locally well generated.

(2) Since $\C$ is accessible, it has splitting idempotents by~\cite[Observation 2.4]{AR}. In particular, all sections in $\C$ have cokernels. To see that $\C$ is efficient, we essentially repeat the arguments in Example~\ref{expl:efficient exact categories}(3). (Ax1) of Definition~\ref{def:efficient exact category} is satisfied by Proposition~\ref{prop:exact category with enough injectives}(2). (Ax2) is true since each object of $\C$ is $\kappa$-presentable for some infinite regular cardinal $\kappa$. Finally, (Ax3) follows from Proposition~\ref{prop:examples of strongly homological sets of maps}(1), using the fact that $\C$ has enough projectives.

(3) Consider the cotorsion pair $(\X,\Y)$ in $\C$ generated by $\clS$. Using the assumption that $\stS = \underline{\Omega \clS} = \underline{\Omega^{-1} \clS}$, a standard dimension shifting argument shows that $\stY = \underline{\Omega \Y}$ and $\stX = \underline{\Omega \X}$. That is, the cotorsion pair is hereditary and cohereditary. Using Proposition~\ref{prop:Hom-orthogonal pairs versus cotorsion pairs}, we see that $(\underline{\Omega \X},\stY) = (\stX,\stY)$ is the Hom-orthogonal pair in $\stC$ generated by $\clS$ (in the sense of Definition~\ref{def:Hom-orthogonal pairs}).

Now on one hand, Corollary~\ref{cor:Distinguished triangle of a Hom-orthogonal pair} together with Proposition~\ref{prop:uniqueness of the triangle} tell us that $\X = \stX = \underline{\Filt\clS}$. On the other hand, since $\stC$ is locally well generated by (1), it follows from the general theory for well generated triangulated categories that $(\Loc\stS, \stS^\perp)$ is a Hom-orthogonal pair; see for instance~\cite[Proposition 3.6]{St} and~\cite[Lemma 3.3]{BIK}. Since $\stX$ is a localizing subcategory of $\stC$ and $\stS \subseteq \stX$, we have $\Loc\stS \subseteq \stX$. However, $(\stX,\stY)$ is generated by $\stS$ as a $\Hom$-orthogonal pair, so we must also have $\stX = \Loc\stS$. Hence $\Loc\stS = \underline{\Filt\clS}$.

(4) This follows easily from (1) and (3) and the fact that Brown representability theorem holds for well generated triangulated categories.
\end{proof}

\begin{rem} \label{rem:Loc S compared to Filt S}
It is very easy to see that $\Loc\stS \subseteq \underline{\Filt\clS}$ for a set $\clS$ of objects of any efficient Frobenius exact category. The intriguing part is the inclusion $\underline{\Filt\clS} \subseteq \Loc\stS$, for which it would be interesting to have a direct proof.
\end{rem}

Now we give a consequence of Theorem~\ref{thm:accessible frobenius exact categories} which is easier to apply directly and immediately implies a result of J\o{}rgensen and Neeman (see Example~\ref{expl:K(Proj)} below). If $\A$ is a category with arbitrary coproducts, $\clS$ a set of objects and $\kappa$ an infinite regular cardinal, we denote by $\Sum_\kappa(\clS)$ the class of all coproducts of families of fewer than $\kappa$ objects from $\clS$. An instance to keep in mind is $\A = \Mod{R}$, $\kappa=\aleph_0$ and $\clS=\{R\}$, in which case $\Sum_\kappa(\clS)$ is the class of all free $R$-modules of finite rank.

\begin{prop} \label{prop:K(B) well generated}
Let $\A$ be an accessible category with coproducts, $\B$ be a class of objects of $\A$ closed under taking coproducts, $\kappa$ be an infinite regular cardinal, and $\clS\subseteq\B$ be a set of $\kappa$-presentable objects. Put $\U = \Cplus{\Sum_\kappa(\clS)}$, the category of all bounded below complexes over $\Sum_\kappa(\clS)$. Then the following assertions are equivalent:
\begin{enumerate}
\item $\B \subseteq \Add(\clS)$,

\item Each complex $X \in \Cpx\B$ is homotopy equivalent to a complex $Y \in \Cpx\A$ which is $\U$-filtered \wrt the componentwise split exact structure.

\item $\Htp\B = \Loc\underline{\U}$.
\end{enumerate}
If the assertions hold, $\Htp\B$ is $\max(\kappa,\aleph_1)$-well generated
by $\U$, and $(\Htp\B, \U^\perp)$ is a Bousfield localizing pair in
$\Htp\A$.
\end{prop}

\begin{proof}
$(2)\Longleftrightarrow (3)$ The category $\Cpx\A$ is accessible by Proposition~\ref{prop: accessibility passes to complexes}. If we consider it with the componentwise split exact structure, it is an efficient Frobenius exact category; see Theorem~\ref{thm:accessible frobenius exact categories}(2) used for $\C = \Cpx\A$. Now, the equivalence is a consequence of Theorem~\ref{thm:accessible frobenius exact categories}(3).

$(1)\Longrightarrow (2)$ Note that if $\Sum(\clS)$ stands for the class of of all coproducts of objects from $\clS$, the inclusions
$\Htp{\Sum(\clS)} \infl \Htp\B \infl \Htp{\Add(\clS)}$ are triangle equivalences by~\cite[I.6.8]{Nee2}. We may, therefore, assume without loss of generality that
$\B=\Sum(\clS)$. Let $X \in \Htp\B$ be a complex
$$
X: \qquad
\cdots \la X^{-1} \overset{\partial^{-1}}\la
X^0 \overset{\partial^0}\la X^1 \la \cdots
\la X^n \overset{\partial^n}\la \cdots,
$$
where $X^n=\coprod_{i\in I_n} S_i^n$ for some index sets $I_n$ and objects $S_i^n \in \clS$. We will inductively construct a $\U$-filtration
$(X_\alpha)_{\alpha<\lambda}$ of $X$ with $X_\alpha$ of the form
$$
X_\alpha: \qquad
\cdots \la \coprod_{i\in I_{-1,\alpha}} S_i^{-1} \la
\coprod_{i\in I_{0,\alpha}} S_i^0 \la
\coprod_{i\in I_{1,\alpha}} S_i^1 \la \cdots,
$$
for suitably chosen subsets $I_{n,\alpha}\subseteq I_n$ ($n\in\mathbb{Z}$) and \st the morphisms $X_\alpha \la X_\beta$ are the canonical split inclusions in all components for each $\alpha<\beta$.

By definition, we must take $I_{n,0} = \emptyset$ for each $n$ and $I_{n,\gamma} = \bigcup_{\alpha<\gamma} I_{n,\alpha}$ for each limit ordinal $\gamma$.
For ordinal successors, assume we have constructed $X_\alpha \subsetneqq X$ for some $\alpha$. Then we choose  $n \in \mathbb{Z}$ \st $I_{n,\alpha}\subsetneqq I_n$ and take $x \in I_n \setminus I_{n,\alpha}$. Having this, we put $I_{m,\alpha+1} = I_{m,\alpha}$ for all $m<n$ and $I_{n,\alpha+1} = I_{n,\alpha} \cup \{x\}$. For $m>n$, we construct $I_{m,\alpha+1}$ by induction on $m$ so that $\partial^{m-1}(\coprod_{i\in I_{m-1,\alpha+1}} S_i^{m-1}) \subseteq \coprod_{i\in I_{m,\alpha+1}} S_i^m$ (with the obvious meaning) and $I_{m,\alpha+1} \setminus I_{m,\alpha}$ is of cardinality $<\kappa$. Note that we can do that since each $S_i^{m-1}$ is $\kappa$-presentable in $\A$ and $I_{m-1,\alpha+1} \setminus I_{m-1,\alpha}$ is of cardinality $<\kappa$ by induction. It follows directly from the construction that $(X_\alpha)_{\alpha<\lambda}$ is a $\U$-filtration in $\Cpx\A$ \wrt the componentwise split exact structure.

$(3)\Longrightarrow (1)$ This is a consequence of (the proof of) \cite[Theorem 2.5]{St}.

Finally, if the assertions (1)--(3) hold, $(\Htp\B, \U^\perp) = (\Htp\B, \Htp\B^\perp)$ is a Bousfield localizing pair in $\Htp\A$ by Theorem~\ref{thm:accessible frobenius exact categories}(4). To prove that $\Htp\B$ is $\max(\kappa,\aleph_1)$-well generated, we may in view of~\cite[2.11 and 2.13(1)]{AR} assume that $\kappa$ is uncountable, and as above, we may also assume that $\B = \Sum(\clS)$. One now proves that $\Htp\B$ is $\kappa$-well generated using a ``componentwise'' argument. Namely, taking a representative set $\D_0$ of $\Sum_\kappa(\clS)$, one readily sees from Definition~\ref{def:well generated} that $\B$ is $\kappa$-well generated by $\D_0$, from which it is not hard to derive that $\Htp\B$ is $\kappa$-well generated by the set $\Cplus{\D_0}$.
\end{proof}

As an immediate corollary, we get a generalization of~\cite[Theorem 5.2]{St}:

\begin{cor} \label{cor:K(B) well generated}
Let $\A$ be an accessible category with arbitrary coproducts and let $\B$ be a class of objects of $\A$ closed under taking coproducts. Then the following assertions are equivalent:
\begin{enumerate}
\item $\Htp{\B}$ is well generated.
\item $\B \subseteq \Add(\clS)$ for some set $\clS \subseteq \B$.
\end{enumerate}
\end{cor}

\begin{proof}
$(2)\Longrightarrow (1)$ This follows from Proposition~\ref{prop:K(B) well generated} by choosing $\kappa$ large enough so that all objects from $\clS$ are $\kappa$-presentable in $\A$.

$(1)\Longrightarrow (2)$ $\Htp{\B}$ being well generated implies that $\Htp{\B} = \Loc \D_0$ for some set $\D_0 \subseteq \Htp{\B}$; see~\cite[Theorems 8.3.3 and 8.4.2]{Nee2}. In particular any object $X \in \B$, viewed as a complex concentrated in degree zero, belongs to $\Loc \D_0$. Let us now take a representative set $\clS \subseteq \B$ for all objects which occur in components of complexes in $\D_0$. The same argument as in the proof of~\cite[Theorem 2.5]{St} shows that $X \in \Add\clS$. Since $X \in \B$ has been chosen arbitrarily, we have proved that $\B \subseteq \Add(\clS)$.
\end{proof}

\begin{exem} \label{expl:K(Proj)}
The following fact was obtained for certain rings in~\cite{Jo} and in full generality in~\cite{Nee3}: Let $R$ be a ring, $\Proj{R}$ the category of projective right $R$-modules and $\U$ the full subcategory of $\Htp{\Proj{R}}$ consisting of bounded below complexes of free modules of finite rank. Then $\Htp{\Proj{R}} = \Loc\U$. It is also proved in~\cite{Nee3} that $\Htp{\Proj{R}}$ is $\aleph_1$-well generated. Both facts are now obtained by putting $\B=\Proj{R}$, $\clS =\{R\}$ and $\kappa=\aleph_0$ in
Proposition~\ref{prop:K(B) well generated}.
\end{exem}

\subsection{Triangulated adjoints without Brown representability}
\label{subsec:adjoints without Brown representability}

A typical problem for triangulated categories is to construct a left or right adjoint to a given triangulated functor $F: \C \la \D$. If $\C$ is well generated in the sense of Definition~\ref{def:well generated}, then $F$ has a right adjoint \iff it preserves coproducts, \cite[8.4.2 and 8.4.4]{Nee2}. This is because of the Brown Representability Theorem \cite[8.3.3]{Nee2}. However, Neeman~\cite{Nee3,Nee} and Murfet~\cite{Mur} recently constructed right adjoints to inclusion functors $\C \infl \D$ where $\C$ typically was not well generated, see~\cite[Theorem 5.2]{St}. Perhaps the simplest example of this type is the embedding $\Htp\F \infl \Htp{\mathrm{Ab}}$, where $\mathrm{Ab}$ is the category of abelian groups and $\F$ the full subcategory of all torsion-free groups (see~\cite[Example 5.3]{St} and Proposition~\ref{prop:adjoints in the homotopy category of a Grothendieck category} below). Their motivation was a connection to Grothendieck duality---a nice overview is given in~\cite{Nee4}. Here we give tools to recover these results as special cases of a general phenomenon.

First, we need to introduce some notation. All through this section $\G$ will be a Grothendieck category and we consider it as an exact category with the abelian exact structure, as in Example~\ref{expl:grothendieck category}. We recall the following notation from \cite{G1}, a part of which we need for a later use:

\begin{nota} \label{not:tilde F and dg-F}
Let $\F$ be a class of objects in $\G$ and put $\T = \F^{\perp_1}$ and $\U = {^{\perp_1} \F}$ (see Section~\ref{subsec:approx}).
\begin{enumerate}
\item We denote by $\tilde\F$ the class of all acyclic complexes $X\in \Cpx\G$ such that $Z^n(X)\in\F$ for all $n\in\bbZ$, where the $Z^n(X) = \Ker \partial^n$ are the cycle objects of $X$.

\item We denote by $\dg\F$ the class of all complexes $X\in \Cpx\F$ \st every chain complex map $X\la Y$, with $Y\in\tilde\T$, is null-homotopic. Here $\tilde\T$ follows the notation of (1) for $\T$ in place of $\F$.

\item Dually, we denote by $\codg\F$ the class of all complexes $X\in \Cpx\F$ \st every map $Y\la X$, with $Y\in\tilde\U$, is null-homotopic.
\end{enumerate}
\end{nota}

Now we give a general statement for construction of right adjoints to inclusions:

\begin{prop} \label{prop:adjoints in the homotopy category of a Grothendieck category}
Let $\G$ be a Grothendieck category, considered with the abelian exact structure, and let $\F \subseteq \G$ be a deconstructible class. Then the following hold:
\begin{enumerate}
 \item The inclusion $\Htp\F \infl \Htp\G$ has a right adjoint.
 \item The inclusion $\underline{\tilde\F }\infl \Htp\G$ has a right adjoint.
 \item If $\F$ contains a generator of $\G$, then also $\underline{\dg\F }\infl \Htp\G$ has a right adjoint.
\end{enumerate}
In particular, $(\Htp\F,\Htp\F^\perp)$,
$(\underline{\tilde\F},\tilde\F^\perp)$ and, if $\F$
contains a generator, then also
$(\underline{\dg\F},\dg\F^\perp)$ are Bousfield
localizing pairs (cf. Definition~\ref{def:t-struct}) in $\Htp\G$.
\end{prop}

\begin{proof}
\cite[Theorem 4.2]{St2} says that under the given assumptions, $\Cpx\F$, $\tilde\F$ and $\dg\F$ are deconstructible in $\Cpx\G$ with the abelian exact structure. Now apply Corollary~\ref{cor:moving between exact structures}, with $\mathbf{E}$ and $\mathbf{E}'$ the classes of componentwise split exact and all exact sequences, respectively.
\end{proof}

Let us illustrate Proposition~\ref{prop:adjoints in the homotopy category of a Grothendieck category} on a few examples, where it allows us to construct recollements (cf.~\cite[\S1.4]{BBD}, \cite[\S9.3]{Nee2} or~\cite{NS} for the concept).

\begin{exem} \label{expl:derived via projectives}
Let $R$ be a ring, $\G = \Mod{R}$ and $\clP = \Proj{R}$ denote the
class of all projective modules. It is well known that
$\underline{\dg\clP} = \Loc\{R\}$ in $\Htp\G$, and that
$\dg\clP^\perp =\tilde\G$, the class of all acyclic
complexes of $R$-modules. It follows from
Proposition~\ref{prop:adjoints in the homotopy category of a
Grothendieck category} that the inclusions $\underline{\dg\clP}\infl
\Htp\G$ and $\tilde\G \infl \Htp\G$ have right adjoints and
$(\underline{\dg\clP},\tilde\G)$ is a Bousfield localizing pair.
This means nothing else than the fact that we have the well known
recollement
$$
\xymatrix@=1.3cm{
\xymoremargin{\tilde\G} \arinfl[r] &
\xymoremargin{\Htp\G} \ar[r] \ar@<1ex>[l] \ar@<-1ex>[l] &
\xymoremargin{\Der\G} \arinfl@<1ex>[l] \arinfl@<-1ex>[l]
}
$$
and the composition $\dg\clP \infl \Htp\G \overset{Q}\la
\Der\G$ is a triangle equivalence.
\end{exem}

\begin{exem} \label{expl:K(Proj) revisited}
Let us now focus on a little less known example
from~\cite{Nee3,Nee}. Let $R$ be a ring, $\clP = \Proj{R}$ and $\F =
\Flat{R}$, the class of all flat right $R$-modules. Then the
inclusion $\Htp\clP \infl \Htp{\Mod{R}}$ has a right adjoint by
Proposition~\ref{prop:adjoints in the homotopy category of a
Grothendieck category} (cf.\ also Example~\ref{expl:K(Proj)}), hence
also $\Htp\clP \infl \Htp\F$ has a right adjoint.
Interpreting~\cite[Theorem 8.6]{Nee3} in our terminology, we learn
that $(\Htp{\clP},\tilde\F)$ is a Bousfield localizing
pair in $\Htp\F$. Using the well-known fact that flat modules form a
deconstructible class in $\Mod{R}$, Proposition~\ref{prop:adjoints
in the homotopy category of a Grothendieck category} yields
existence of a right adjoint to $\underline{\tilde\F}\infl
\Htp{\Mod{R}}$. In particular, $\tilde\F\infl \Htp\F$
has a right adjoint. To summarize, we have obtained a recollement
$$
\xymatrix@=1.3cm{ \xymoremargin{\tilde\F} \arinfl[r] &
\xymoremargin{\Htp\F} \ar[r] \ar@<1ex>[l] \ar@<-1ex>[l] &
\xymoremargin{\Htp\F / \tilde\F} \arinfl@<1ex>[l]
\arinfl@<-1ex>[l] }
$$
\st the composition $\Htp\clP \infl \Htp\F \overset{Q}\la \Htp\F /
\tilde\F$ is a triangle equivalence.
\end{exem}

Before discussing a generalization of this example for schemes, we give a partial answer to a question in~\cite[Remark B.7]{Mur} as to when the class of flat quasi-coherent sheaves is closed under products. We prove this for {\bf Dedekind schemes}, that is, for noetherian integral schemes whose rings of regular functions on affine open subsets are Dedekind domains, \cite[(7.13), p. 188]{GW}. Note that any non-singular connected curve (affine or projective, cf.~\cite[\S I.6]{Hart2}) gives rise to a Dedekind scheme, see~\cite[Example 15.2]{GW}.

\begin{prop} \label{prop:sheaves closed under products}
Let $\bbX$ be a Dedekind scheme. Then the class $\F = \Flat\bbX$ of flat quasi-coherent sheaves over $\bbX$ is closed under taking products in $\Qco\bbX$.
\end{prop}

\begin{proof}
Let us recall a few standard facts: By~\cite[Exercise II.5.15(e)]{Hart2}, $\Qco\bbX$ is an $\aleph_0$-accessible category and the subcategory of $\aleph_0$-presentable sheaves is precisely $\coh\bbX$, the category of coherent sheaves. Moreover, each coherent sheaf satisfies the ascending chain condition on subobjects by~\cite[Proposition II.5.4]{Hart2}. In other words, $\Qco\bbX$ is what one calls a locally noetherian Grothendieck category.

Further, for each closed point $x \in \bbX$ there is a (unique) sheaf $S_{\{x\}}$ whose stalk at $x$ is $\OX{\bbX,x}/\mathfrak{m}_{X,x}$ and such that $(S_{\{x\}})_y = 0$ for each $y \in \bbX \setminus \{x\}$. It is immediate that $S_{\{x\}}$ a simple coherent sheaf supported at $\{x\}$, a so called skyscraper sheaf in~$x$ (cf.~\cite[Exercise II.1.17]{Hart2} or~\cite[Exercise 2.14]{GW}). We now claim the following:
\begin{enumerate}
 \item $M \in \Qco\bbX$ is flat \iff the stalk $M_x$ is a torsion-free $\OX{\bbX,x}$-module for each closed point $x \in \bbX$;
 \item The class $\vect\bbX$ of locally free coherent sheaves is preenveloping in $\coh\bbX$.
\end{enumerate}

To see (1), note that $M$ is flat \iff all stalks at closed points are flat by~\cite[Theorem 7.1]{Mat}. Since all such stalks $\OX{\bbX,x}$ are discrete valuation domains, $M_x$ is flat \iff it is torsion-free. For (2), note that $M \in \coh\bbX$ is locally free \iff $\Hom_{\coh\bbX}(S_{\{x\}},M) = 0$ for all closed points $x$. Indeed, each $S_{\{x\}}$ is supported at a single closed point, so this is a local question. The corresponding fact for affine Dedekind schemes follows from the isomorphisms:
$$
\Hom_{R_\mathfrak{m}}(R/\mathfrak{m}, M_\mathfrak{m}) \cong
\big(\Hom_R(R/\mathfrak{m}, M)\big)_\mathfrak{m} =
\Hom_R(R/\mathfrak{m}, M),
$$
where $R$ is a Dedekind domain and $\mathfrak{m} \subseteq R$ is maximal; see~\cite[Theorem 7.11]{Mat}. Given $M \in \coh\bbX$ and denoting by $tM$ the unique maximal $\clS$-filtered coherent subsheaf of $M$, where $\clS = \{S_{\{x\}} \mid x \in \bbX \textrm{ closed} \}$, we get an epimorphism $p_M: M \defl M/tM$, where $M/tM \in \vect\bbX$ by~\cite[Theorem 7.12]{Mat} and $p_M$ is a $\vect\bbX$-preenvelope. In fact, $p_M$ usually splits; see~\cite[p. 127]{AHK}. This proves the claim.

Finally, it follows from (1) that $M$ is flat \iff each coherent subsheaf is flat and, by~\cite[Theorem 7.12]{Mat}, a coherent sheaf is flat if and only if it is locally free. Therefore $\F$ is the class of direct limits of locally free coherent sheaves. Invoking~\cite[Theorem 4.2]{CB}, we learn that $\F$ is closed under products in $\Qco\bbX$ \iff $\vect\bbX$ is preenveloping in $\coh\bbX$, which has been proved in~(2).
\end{proof}

\begin{exem} \label{expl:mock homotopy categories of projectives}
A generalization of Example~\ref{expl:K(Proj) revisited} for non-affine schemes was first studied by Murfet~\cite{Mur}. Suppose we have a scheme $\bbX$ and denote by $\F = \Flat\bbX$ the class of all flat quasi-coherent sheaves. Since the class of flat modules is deconstructible over any commutative ring and deconstructibility of a class of sheaves can be tested locally by (the proof of) \cite[Theorem 3.8]{EGPT} (see also~\cite{St2}), we infer that $\F$ is deconstructible in $\Qco\bbX$. In particular, the inclusion $\tilde\F\infl \Htp{\Qco\bbX}$, and also $\tilde\F \infl \Htp\F$, have right adjoints by Proposition~\ref{prop:adjoints in the homotopy category of a Grothendieck category}. This provides us with a Bousfield localizing pair $(\tilde\F,\tilde\F^\perp)$ in $\Htp\F$. Hence, the Verdier quotient $\Hmock{\Proj\bbX} = \Htp\F / \tilde\F$, for which Murfet coined the term `mock homotopy category of projectives', is well behaved. For instance it is triangle equivalent to $\tilde\F^\perp$, and we immediately obtain a so-called localization sequence (cf.~\cite[Theorem 3.16]{Mur})
$$
\xymatrix@=1.3cm{ \xymoremargin{\tilde\F}
\arinfl@<.5ex>[r] & \xymoremargin{\Htp\F} \ar@<.5ex>[r]
\ar@<.5ex>[l] & \xymoremargin{\Hmock{\Proj\bbX}} \arinfl@<.5ex>[l] }
$$
Note, however, that we cannot expect to get a recollement here as in Example~\ref{expl:K(Proj) revisited}. Indeed, the inclusion $\Htp\F \infl \Htp{\Qco\bbX}$ has a left adjoint for a Dedekind scheme~$\bbX$ by the forthcoming Corollary~\ref{cor:left adjoint deconstructible}, since $\F$ is deconstructible and closed under products by Proposition~\ref{prop:sheaves closed under products}. If the mentioned recollement existed, the inclusion $\tilde\F \infl \Htp{\Qco\bbX}$ would also have a left adjoint, which would in particular imply that $\tilde\F$ would be closed under taking products in $\Htp{\Qco\bbX}$. Then a product of epimorphisms between flat quasi-coherent sheaves over $\bbX$ would have to be an epimorphism again. This is known to be false for $\bbX = \PoneC$, see \cite[Example 4.9]{Kr5}. One can prove even more, see~\cite[Corollary A.14]{Mur}.

Using this general principle, we can also construct adjoints to inclusions of other remarkable classes of sheaves. Let $\V = \Vect\bbX$ be the class
of all locally projective quasi-coherent sheaves. Using \cite[Theorem 3.8]{EGPT} and Proposition~\ref{prop:adjoints in the homotopy category of a Grothendieck category}
as above, we get a localization sequence
$$
\xymatrix@=1.3cm{ \xymoremargin{\tilde\V}
\arinfl@<.5ex>[r] & \xymoremargin{\Htp\V} \ar@<.5ex>[r]
\ar@<.5ex>[l] & \xymoremargin{\Htp\V / \tilde\V}
\arinfl@<.5ex>[l] }
$$
Now observe that $\underline{\tilde\V} = \Htp\V \cap \tilde\F$ in $\Htp\F$, since we have this equality for any affine scheme by~\cite[2.14 and 2.15]{Nee3} and all involved classes are defined by local properties (cf.~\cite{RG} and~\cite[\S2]{Dr}). This yields a natural triangulated
functor
$$ \Htp\V / \tilde\V \la \Htp\F / \tilde\F = \Hmock{\Proj\bbX}. $$
The obvious question is: Can this functor be an equivalence? Is it an equivalence for instance when $\bbX$ is a quasi-projective variety and we have enough vector bundles (see~\cite[Lemma 2.1.3]{TT})? If so, we would have a presentation of $\Hmock{\Proj\bbX}$ in terms of vector bundles.
\end{exem}

So far, we have been concerned only with right adjoints and precovers. Let us briefly discuss the dual situation, where again we have two options. First, if we have a product preserving triangulated functor $F: \C \la \D$ and $\C$ is compactly generated, then the existence of a left adjoint to $F$ follows from~\cite[Theorem 8.6.1]{Nee2} and~\cite[\S1.2]{Kr4}. Second, Proposition~\ref{prop:Neeman's} dualizes smoothly, so we can, under some circumstances, construct a left adjoint using preenvelopes.

\begin{prop} \label{prop:left adjoint}
Let $\E$ be an efficient exact category and let $\Y \subseteq \E$ be a class of objects. Then the inclusion $\Htp\Y \infl \Htp\E$ has a left adjoint and $({^\perp \Htp\Y}, \Htp\Y)$ is a Bousfield localizing pair in $\Htp\E$ provided  one of the following two conditions holds:
\begin{enumerate}
\item There is a set of objects $\clS \subseteq \E$ \st $\Y = \clS^{\perp_1}$.
\item $\Y$ is preenveloping and closed under taking coproducts in $\E$.
\end{enumerate}
In such a case, even if $\Y$ does not have all coproducts, the homotopy category $\Htp\Y$ does. 
\end{prop}

\begin{proof}
We have two exact structures on $\Cpx\E$, one defined by the class $\mathbf{E}$ of all componentwise split short exact sequences in $\E$, and the other defined by the class $\mathbf{E}'$ of sequences of complexes which are conflations in $\E$ in each component. We denote the corresponding exact categories by $\Cpx\E_\mathbf{E}$ and $\Cpx\E_{\mathbf{E}'}$, respectively. Then both $\Cpx\E_\mathbf{E}$ and $\Cpx\E_{\mathbf{E}'}$ are as in Setup~\ref{setup:nice exact category}; see Example~\ref{expl:semisplit structure}.

Assume first that $\Y = \clS^{\perp_1}$ in $\E$ for some set $\clS$. For each $X \in \E$, let $\hat X$ denote the complex concentrated in degrees $0$ and $1$
$$
\hat X: \qquad
\cdots \la 0 \la 0 \la X \overset{1_X}\la X \la 0 \la 0 \la \cdots,
$$
and denote by $\hat{\clS}$ the set of complexes of the form $\hat{X}[n]$, with $X\in\clS$ and $n\in\mathbb{Z}$. Now fix a strongly homological set $\I$ of inflations in $\E$ \st (up to isomorphism) $\Coker(\I)=\clS$, and consider the set $\hat{\I}$ of obvious chain maps $\hat{i}[n]: \hat{A}[n] \la \hat{B}[n]$, with $i:A \infl B$ in $\I$ and $n\in\mathbb{Z}$. It is routine to check that $\hat{\I}$ is strongly homological in $\Cpx\E_{\mathbf{E}'}$ and $\Coker(\hat{\I})=\hat{\clS}$. It follows from Theorem~\ref{thm:orthogonality of morphisms - cotorsion pairs}(4) used for $\C = \Cpx\E_{\mathbf{E}'}$ that $\hat{\clS}^{\perp_1}$ is preenveloping and we claim that $\hat{\clS}^{\perp_1}=\Cpx\Y$ in $\Cpx\E_{\mathbf{E}'}$. Indeed, the inclusion $\Cpx\Y \subseteq \hat{\clS}^{\perp_1}$ follows from the equalities and isomorphisms
$$
\Ext_{\Cpx\E_{\mathbf{E}'}}^1(\hat{X}[n],Y) = \Ext_{\Cpx\E_\mathbf{E}}^1(\hat{X}[n],Y) \cong \stHom_{\Cpx\E_\mathbf{E}}(\hat{X}[n-1],Y) = 0,
$$
for each $X \in \clS$ and $Y \in \Cpx\Y$. Conversely, the inclusion $\Cpx\Y \supseteq \hat{\clS}^{\perp_1}$ is implied by the following isomorphism for each $X \in \clS$ and $Y \in \Cpx\Y$, see~\cite[Lemma 3.1(5)]{G3}:
$$
\Ext_{\Cpx\E_{\mathbf{E}'}}^1(\hat{X}[-n],Y) \cong \Ext_\E^1(X,Y^n).
$$
The claim is proved and consequently $\Htp\Y$ is preenveloping in $\Htp\E$. It only remains to use the dual statement to Proposition~\ref{prop:Neeman's}.

Assume next that $\Y$ is preenveloping and closed under taking coproducts in $\E$. Using $\Y$-preenvelopes and pushouts, one shows by the same argument as for~\cite[Lemma I.4.6(i)]{Hart} that each $X \in \Cplus\E$ admits a $\Cplus\Y$-preenvelope concentrated in the same degrees as $X$. Hence the dual of Proposition~\ref{prop:Neeman's} tells us that  $\big({^\perp \big(\Hplus\Y}\big), \Hplus\Y\big)$ is a Bousfield localizing pair in $\Hplus\E$.

Let now $X\in \Cpx\E$ be arbitrary. Then $X$ can be expressed in $\Htp\E$ as the homotopy colimit (see~\cite[\S1.6]{Nee2} for details) of the chain $\sigma^{\geq 0} X \infl \sigma^{\geq -1} X \infl \cdots \infl \sigma^{\geq -n} X \infl \cdots$ of its brutal (also known as stupid) truncations.
%
%
By the previous paragraph, we have for each $n \ge 0$ a triangle
$$
\tau_n: \qquad
F_n \la \sigma^{\geq -n} X \la Y_n \laplus
$$
in $\Htp\E$ with $F_n \in {^\perp \Htp\Y}$ and $Y_n \in \Hplus\Y$, and the Bousfield localizing pair gives us a unique chain of maps of triangles $\tau_0 \la \tau_1 \la \cdots \la \tau_n \la \cdots$ extending the chain of truncations. By the octahedral axiom, we get a commutative diagram with triangles in rows and columns:
$$
\begin{CD}
\coprod_{n \ge 0} F_n @>>> \coprod_{n \ge 0} \sigma^{\geq -n} X @>>> \coprod_{n \ge 0} Y_n @>+>>    \\
@V{1-\mathrm{shift}}VV              @V{1-\mathrm{shift}}VV                        @VVV            \\
\coprod_{n \ge 0} F_n @>>> \coprod_{n \ge 0} \sigma^{\geq -n} X @>>> \coprod_{n \ge 0} Y_n @>+>>  \\
         @VVV                                @VVV                                 @VVV            \\
           F          @>>>                     X                @>>>                Y      @>+>>  \\
         @V+VV                               @V+VV                                @V+VV           \\
\end{CD}
$$
where clearly $F \in {^\perp \Htp\Y}$ and $Y \in \Htp\Y$. Hence
$({^\perp \Htp\Y}, \Htp\Y)$ is again a Bousfield localizing pair in
$\Htp\E$.

The final assertion is standard. Namely, $\Htp\E$ has coproducts by Lemma~\ref{lem:existence of exact coproducts} and a left adjoint functor takes coproducts to coproducts.
\end{proof}


Let us provide a few examples for the existence of a left adjoint functor.

\begin{exem} \label{expl:derived via injectives}
Let $\G$ be a Grothendieck category and $\I$ the class of all
injective objects in $\G$. We will see in Theorem~\ref{thm:model
structures on Grothendieck categories} that $(\tilde\G,\codg\I)$ is
a complete cotorsion pair in $\Cpx\G$ with the abelian exact
structure. In particular $\codg\I$ is preenveloping in $\Cpx\G$ and
$\underline{\codg\I} \infl \Htp\G$ has a left adjoint by
Proposition~\ref{prop:Neeman's}. It is well known that in this case
we have the Bousfield localizing pair
$(\tilde\G,\underline{\codg\I})$ and the composition
$\underline{\codg\I} \infl \Htp\G \overset{Q}\la \Der\G$ is a
triangle equivalence.
\end{exem}

\begin{exem} \label{expl:Inj}
Let $\G$ be a Grothendieck category and $\I$ the class of injective objects. Recall that there exists a set $\clS$ of objects of $\G$ \st $\I = \clS^{\perp_1}$. It follows from Proposition~\ref{prop:left adjoint} that the inclusion $\Htp\I \infl \Htp\G$ always has a left adjoint and that $\Htp\I$ has coproducts, even when $\G$ is not locally Noetherian.
\end{exem}

We end the section with another application of Proposition~\ref{prop:left adjoint}, which can be applied for instance to $\G = \Mod{R}$ and $\F = \Flat{R}$ for a left coherent ring $R$, or to $\G = \Qco\bbX$ and $\F = \Flat\bbX$ as in Proposition~\ref{prop:sheaves closed under products}.

\begin{cor} \label{cor:left adjoint deconstructible}
Let $\G$ be a Grothendieck category and $\F$ be a deconstructible class in $\G$. Then the following assertions are equivalent:
\begin{enumerate}
\item Every product of objects of $\F$ is a direct summand of an object of $\F$.
\item $\Htp\F$ is a localizing subcategory of $\Htp\G$ and the inclusion functor $\Htp\F \infl \Htp\G$ has a left (and a right) adjoint.
\end{enumerate}
\end{cor}

\begin{proof}
(2) $\implies$ (1) If $X \in \G$ and $\eta_X: X \la Y$ with $Y \in \Htp\F$ is the map coming from the unit of the adjunction, then induced morphism $\eta_X^0: X \la Y^0$ in degree zero is an $\F$-preenvelope in $\G$. Indeed, if $f': X \la F$ is any morphism in $\G$ with $F \in \F$, then the adjunction property gives in $\Htp\G$ a unique morphism $g: Y \la F$ \st $f' = g \circ \eta_X$. It follows that $f' = g^0 \circ \eta_X^0$ in $\G$ and, hence, $\eta_X^0$ is an $\F$-preenvelope. Therefore, $\F$ is preenveloping in $\G$. Now notice that an  $\F$-preenvelope of a product of objects of $\F$ must be a section, so (1) follows.

(1) $\implies$ (2) By~\cite[Corollary 2.7]{St2}, there is a regular cardinal $\kappa = \kappa(\G)$ with the following property: For each morphism $f: X \la F$ in $\G$ with $F \in \F$ and $X$ $\lambda$-presentable for a regular cardinal $\lambda$, there is a $\max(\kappa,\lambda)$-presentable subobject $F' \subseteq F$ \st $F' \in \F$ and $\Im f \subseteq F'$. It follows that if we fix $X \in \G$, there is a set $\clS \subseteq \F$ \st any morphism $f: X\longrightarrow F$ with $F \in \F$ factors through an inclusion $F' \subseteq F$ for some $F' \in \clS$. By our assumption there is $F'' \in \G$ \st $\prod_{F \in \clS} F^{\Hom_\G(X,F)} \oplus F'' \in \F$, and it follows that the obvious morphism
$$
h: X \la \prod_{F \in \clS} F^{\Hom_\G(X,F)} \oplus F''
$$
is an $\F$-preenvelope of $X$. Since $\F$ is also closed under coproducts by Corollary~\ref{cor:filtrations closed under transf ext}, Proposition~\ref{prop:left adjoint} implies that $\Htp\F \infl \Htp\G$ has a left adjoint. Moreover, $\Htp\F \infl \Htp\G$ has a right adjoint by Proposition~\ref{prop:adjoints in the homotopy category of a Grothendieck category}.
\end{proof}

\subsection{Constructing model structures for Grothendieck categories}
\label{subsec:constructing monoidal model structures}

Another area where cotorsion pairs and approximations are useful, is the construction of model category structures on abelian categories, \cite{H1,H3}. The particular case of recent interest was the construction of the derived category of the category $\Qco\bbX$ of quasi-coherent sheaves with the tensor product $\otimes$. The main incentive of Gillespie in his work~\cite{G3,G2,G1} was to give a clean and general framework to deal with the derived functor of $\otimes$. Generalizing and streamlining arguments from the subsequent work of Estrada, Guil, Prest and Trlifaj~\cite{EGPT} we have the following main statement, where we use Notation~\ref{not:tilde F and dg-F}:

\begin{theor} \label{thm:model structures on Grothendieck categories}
Let $\G$ be a Grothendieck category and $\F \subseteq\ \G$ be \st
\begin{enumerate}
 \item $\F$ is deconstructible in $\G$ with the abelian exact structure,
 \item $\F$ is closed under taking kernels of epimorphisms and summands in $\G$, and
 \item $\F$ contains a generator of $\G$.
\end{enumerate}
If we put $\T = \F^{\perp_1}$ in $\G$, then $(\dg\F,\tilde\T)$ and $(\tilde\F,\codg\T)$ are complete cotorsion pairs in $\Cpx\G$ with the abelian exact structure. Moreover, there is a model category structure (in the sense of~\cite{H2}) on $\Cpx\G$ \st
\begin{itemize}
 \item Cofibrations (resp. trivial cofibrations) are precisely monomorphisms whose cokernels are in $\dg\F$ (resp. $\tilde\F$).
 \item Fibrations (resp. trivial fibrations) are precisely epimorphisms whose kernels are in $\codg\T$ (resp. $\tilde\T$).
 \item Weak equivalences are precisely quasi-isomorphisms.
\end{itemize}
In particular, the homotopy category of $\Cpx\G$ is precisely the derived category $\Der\G$.
\end{theor}

\begin{rem} \label{rem:explanation for module structures}
Note that Theorem~\ref{thm:model structures on Grothendieck categories} generalizes both~\cite[Theorem 4.12]{G1} (since any deconstructible class is Kaplansky due to~\cite[Corollary 2.7]{St2}) and the first paragraph of~\cite[Theorem 4.4]{EGPT}. As an illustration, for any Grothendieck category one can put $\F = \G$, which results in the construction of $\Der\G$ using injectives.
If $\bbX$ is a quasi-compact and separated scheme, we can put $\G = \Qco\bbX$ and $\F = \Flat\bbX$ by~\cite[Corollary 3.22]{Mur}. Using~\cite[Lemma 2.1.3]{TT}, the theorem applies to $\G = \Qco\bbX$ and $\F = \Vect\bbX$ if $\bbX$ is a quasi-projective variety.
\end{rem}

\begin{rem} \label{rem:explanation for monoidal module structures}
The main motivation for seeking alternatives to injectives in the construction of $\Der{\Qco\bbX}$ is that we wish the resulting model structure to be compatible with the tensor product of complexes of sheaves. For this we need cofibrant replacements to be complexes of flat sheaves. We refer to~\cite{EGPT} for details.
\end{rem}

The proof will be heavily based on Gillespie's ideas and follows the outline in~\cite[\S7.2]{H3} with replacements made only in a few necessary spots. One such place involves resolving Hovey's comment on a generalization of hereditary cotorsion pairs for Grothendieck categories, see~\cite[\S7.2, p. 292]{H3}.


\begin{lemma} \label{lem:hereditary cotorsion pairs in Grothendieck categories}
Let $\G$ be a Grothendieck category and $(\F,\T)$ a cotorsion pair
in $\G$ \st $\F$ contains a generator. Then the following assertions
are equivalent:
\begin{enumerate}
 \item $\Ext_\G^n(F,T) = 0$ for each $F \in \F$, $T \in \T$ and $n \ge 1$.
 \item $\Ext_\G^2(F,T) = 0$ for each $F \in \F$ and $T \in \T$.
 \item $\F$ is closed under taking kernels of epimorphisms.
 \item $\T$ is closed under taking cokernels of monomorphisms.
\end{enumerate}
\end{lemma}

\begin{proof}
The implications (1) $\implies$ (2) $\implies$ (3) and (2) $\implies$ (4) are easy. Let us focus on (3) $\implies$ (1). Assume that $n \ge 2$ and the exact sequence
$$ \varepsilon: \qquad 0 \to T \la E_1 \la E_2 \la \dots \la E_n \la F \to 0 $$
represents an element of $\Ext_\G^n(F,T)$. Using the dual version of~\cite[Lemma I.4.6(i)]{Hart}, one sees that there is a quasi-isomorphism from a complex of the form
$$
\dots \la F^0 \overset{\partial^0}\la
F^1 \overset{\partial^1}\la
F^2 \overset{\partial^2}\la \dots
\overset{\partial^{n-1}}\la F^n \la
0 \la 0 \la \dots
$$
with all $F_i \in \F$ to the complex $(E_1 \la E_2 \la \dots \la E_n)$ concentrated in degrees $1$ to $n$. One easily obtains the commutative diagram with exact rows:
$$
\xymatrix{ \varepsilon: & 0 \ar[r] & T \ar[r] & E_1 \ar[r] & E_2
\ar[r] & \dots \ar[r] & E_n \ar[r] & F \ar[r] & 0\phantom{,}
\\
\eta: & 0 \ar[r] & T \ar[r] \ar@{=}[u] & \Coker \partial^0 \ar[r]
\ar[u] & F_2 \ar[r]^{\partial^2} \ar[u] & \dots
\ar[r]^{\partial^{n-1}} & F_n \ar[r] \ar[u] & F \ar[r]  \ar@{=}[u] &
0, }
$$
from which it follows that $\varepsilon$ and $\eta$ represent the same element in $\Ext_\G^n(F,T)$. Moreover, assuming (3) we can easily prove by induction that $\Im \partial^i \in \F$ for all $1 \le i \le n-1$. Therefore, $0 \to T \la \Coker \partial^0 \la \Im \partial^1 \to 0$ splits and both $\varepsilon$ and $\eta$ represent zero in the $\Ext$ group. Finally, the proof of (4) $\implies$ (1), where we employ the fact that $\T$ always contains an injective cogenerator, is dual.
\end{proof}

We are in a position to give a proof of Theorem~\ref{thm:model structures on Grothendieck categories} now.

\begin{proof}[Proof of Theorem~\ref{thm:model structures on Grothendieck categories}]
Let $\F \subseteq \G$ be a class as in the statement and $\T = \F^{\perp_1}$ in $\G$. In particular, there is a generating set $\clS$ \st $\F = \Filt\clS$ in the sense of Definition~\ref{def:filtrations}, and $\T = \clS^{\perp_1}$ by Theorem~\ref{thm:orthogonality of morphisms - cotorsion pairs}(3). It follows from Corollary~\ref{cor:consequences of Theorem}(2) that $(\F,\T)$ is a complete cotorsion pair in $\G$. Invoking~\cite[Proposition 3.6]{G3}, we readily infer that $(\dg\F,\tilde\T)$ and $(\tilde\F,\codg\T)$ are cotorsion pairs in $\Cpx\G$ with the abelian exact structure. Here we use the fact that $\Cpx\G$ is a Grothendieck category, and that both $\dg\F$ and $\tilde\F$ are generating classes of $\Cpx\G$. More precisely, if $F \in \F$ is a generator of $\G$, then the complexes
$$
\dots \la 0 \la 0 \la F \overset{1_F}\la F \la 0 \la 0 \la \dots
$$
form a set of generators of $\Cpx\G$ which is both in $\tilde\F$ and $\dg\F$.  Assuming the deconstructibility of $\F$ in $\G$, both $\dg\F$ and $\tilde\F$ are deconstructible in $\Cpx\G$ by~\cite[Theorem 4.2]{St2}. Using Theorem~\ref{thm:orthogonality of morphisms - cotorsion pairs}(3) and Corollary~\ref{cor:consequences of Theorem}(2) again, it follows that $(\dg\F,\tilde\T)$ and $(\tilde\F,\codg\T)$ are complete.

Finally, the existence of the model structure follows directly from~\cite[Theorem 2.2]{H1}, provided we can prove the following two equalities:
$$
\dg\F \cap \W = \tilde\F \qquad \textrm{and} \qquad \codg\T \cap \W
= \tilde\T,
$$
where $\W \subseteq \Cpx\G$ is the class of all acyclic complexes. With help of Lemma~\ref{lem:hereditary cotorsion pairs in Grothendieck categories}, which tells us that $\T$ is closed under cokernels of monomorphisms, the equalities follow from~\cite[Corollary 3.9]{G1}. In the argument there, one just has to keep in mind that we already know that $(\dg\F,\tilde\T)$ is a complete cotorsion pair in $\Cpx\G$.
\end{proof}

\appendix
\section{More on exact categories}
\label{sec:appendix exact catg}

Throughout this appendix, let $\C$ be an exact category in the sense of \cite{Q,Kst}. We shall list and prove a few statements that we have used earlier in the paper. We will prove them from the axioms, since we work with exact categories which are typically not small, so it is not formally correct to use the embedding theorem from~\cite[Appendix A]{Kst} or~\cite[A.7.1 and A.7.16]{TT}. All the statements also have dual versions, the formulation of which is left to the reader. Let us start with a simple fact:

\begin{lemma} \label{lem:adding a deflation}
Let $p:X\la Z$ and $f:Y\la Z$ be morphisms in $\C$, where $p$ is a deflation. The induced morphism $\begin{pmatrix}p & f\end{pmatrix}: X\oplus Y\la Z$ is a deflation.
\end{lemma}
\begin{proof}
Since the pullback of $p$ and $f$ exists by the axioms, the result is a consequence of the dual of \cite[Proposition 2.12]{Bu}.
\end{proof}

A more interesting problem is to determine when a converse statement holds, that is, when $\begin{pmatrix}p & f\end{pmatrix}$ being a deflation
implies that $p$ is a deflation. To do so, we have to be careful regarding existence of direct summands. Recall that given morphisms $r: X \la Y$ and $s: Y \la X$ \st $rs = 1_Y$, we call $s$ a section and $r$ a retraction. In the sequel, we will need the condition that every section has a cokernel or, equivalently by~\cite[Lemma 7.1]{Bu}, every retraction has a kernel. Note that in such a case, every section is an inflation and every
retraction is a deflation by~\cite[Corollary 7.5]{Bu}.

\begin{lemma} \label{lem:(p,pu) deflation implies f deflation}
Suppose that every section in $\C$ has a cokernel. Let $p:X\la Z$ and $u:Y\la X$ be any morphisms. The morphism $\begin{pmatrix}p & pu\end{pmatrix}:X\oplus Y\la Z$ is a deflation \iff $p$ is a deflation.
\end{lemma}

\begin{proof}
The ``if'' part follows from Lemma \ref{lem:adding a deflation}. Conversely, suppose that
$\begin{pmatrix}p & pu\end{pmatrix}:X\oplus Y\defl Z$ is a deflation. First we reduce the problem to the case $u=0$.
For this, note that the endomorphism of $X\oplus Y$ given by
$$
\smallpmatrix{
1_X & u   \\
0   & 1_Y
}: X\oplus Y\la X\oplus Y
$$
is an isomorphism, and $\begin{pmatrix} p & pu\end{pmatrix} =
\begin{pmatrix}p & 0\end{pmatrix} \circ \smallpmatrix{1_X & u \\ 0 &
1_Y}$.

Therefore, we are left to prove that if $\begin{pmatrix}p & 0\end{pmatrix}:X\oplus Y\defl Z$ is a deflation, then so is $p$. Considering the conflation
$$ 0 \to K \overset{\smallpmatrix{u \\ v}}\la X\oplus Y \overset{\smallpmatrix{p & 0}}\la Z \to 0 $$
and the fact that the composition
$$
Y \overset{\smallpmatrix{0\\ 1}}\la X\oplus Y
\overset{\smallpmatrix{p & 0}}\la Z
$$
vanishes, it follows that there exists a unique morphism $i:Y\la K$ \st $\smallpmatrix{u\\ v} i=\smallpmatrix{0\\ 1}$. That is, $i$ is a section, the cokernel $q: K \la C$ of $i$ is a retraction, and we have $s: C \la K$ \st $qs = 1_C$ and $iv + sq = 1_K$. A short computation gives
$$ vsq = v(1_K - iv) = v - viv = v - 1_Y v = 0. $$
Then also $vs = 0$ since $q$ is an epimorphism. It follows by a standard argument that $\begin{pmatrix}s & i\end{pmatrix}: C \oplus Y \la K$ is an isomorphism and if we put $k = u \circ s$, we have a conflation
$$
0 \to
C \oplus Y \overset{\smallpmatrix{k & 0 \\ 0 & 1}}\la
X \oplus Y \overset{\smallpmatrix{p & 0}}\la
Z \to
0.
$$
By constructing a pushout of this conflation along the projection $C \oplus Y \defl C$, we get the following commutative diagram whose lower row is a conflation:
$$
\xymatrix{
0 \ar[r] &
C \oplus Y \ar[r]^{\smallpmatrix{k & 0 \\ 0 & 1}} \ar[d]_{\smallpmatrix{1 & 0}} &
X \oplus Y \ar[r]^{\smallpmatrix{p & 0}} \ar[d]_{\smallpmatrix{1 & 0}} &
Z \ar[r] \ar@{=}[d] &
0
\\
0 \ar[r] &
C \ar[r]^{k} &
X \ar[r]^{p} &
Z \ar[r] &
0
}
$$
Hence, $p$ is a deflation as desired.
\end{proof}

Before pushing this idea further, let us give a definition:

\begin{defi} \label{def:weakly terminal subfamily}
Let $(f_i:X_i\la Y)_{i\in I}$ be a family of morphisms in $\C$ with the same codomain. A subfamily $(f_j)_{j\in J}$ indexed by a subset $J \subseteq I$ is called {\bf weakly terminal} if for every index $i\in I$, there exists an index $j\in J$ such that $f_i$ factors through $f_j$.
\end{defi}

\begin{rem} \label{rem:weakly terminal subfamily}
The terminology comes from the fact that, if we consider $\clS = \{f_i \mid i \in I\}$ as a full subcategory of the comma category $\C/Y$, then the condition of Definition~\ref{def:weakly terminal subfamily} precisely says that $\{f_j \mid j \in J\}$ is a weakly terminal set of objects of $\clS$.
\end{rem}

\begin{lemma} \label{lem:weakly terminal deflation}
Suppose that $\C$ has arbitrary coproducts and every section in $\C$ has a cokernel. Suppose also that $(f_j:X_j\la Y)_{j\in J}$ is a weakly terminal subfamily of
a family of morphisms $(f_i:X_i\la Y)_{i\in I}$. Then the morphism
$$ (f_i): \coprod_{i\in I} X_i \la Y $$
is a deflation \iff the following morphism is a deflation
$$ (f_j): \coprod_{j\in J} X_j \la Y. $$
\end{lemma}

\begin{proof}
For each subset $K\subseteq I$, let us put $X_K = \coprod_{i\in K} X_i$ and denote by $f_K: X_K\la Y$ the morphism such that $f_K\lambda_i=f_i$ for each $i\in K$, where $\lambda_i:X_i\la X_K$ denotes the coproduct inclusion.

By hypothesis, for each $i\in I\setminus J$ there exist an index $j\in J$ and a morphism $v_i:X_i\longrightarrow X_j$ \st $f_j v_i=f_i$. We shall denote by $u_i$ the composition
$$ X_i\overset{v_i}\la X_j\overset{\lambda_j}\la X_J $$
and consider the morphism $u:X_{I\setminus J}\la X_J$ determined by $u\lambda_i=u_i$, for all $i\in I\setminus J$. An easy computation reveals the equalities
$$ f_Ju\lambda_i = f_Ju_i = f_J\lambda_jv_i = f_jv_i = f_i = f_{I \setminus J} \lambda_i $$
for all $i\in I\setminus J$, so $f_Ju=f_{I\setminus J}$. Finally, we invoke Lemma~\ref{lem:(p,pu) deflation implies f deflation} which says that $f_I=\begin{pmatrix}f_J & f_Ju\end{pmatrix}: X_J\oplus X_{I\setminus J} \la Y$ is a deflation \iff $f_J: X_J \la Y$ is a deflation.
\end{proof}

We conclude with a proposition which says that coproducts and transfinite composition of inflations are exact in $\C$ provided we have enough injectives and cokernels of sections. Compare this to Setup~\ref{setup:nice exact category} and Lemma~\ref{lem:existence of exact coproducts}.

\begin{prop} \label{prop:exact category with enough injectives}
Suppose that $\mathcal{C}$ has enough injectives and that every section in $\mathcal{C}$  has a cokernel. Then the following assertions hold:
\begin{enumerate}
\item A morphism $f:X\la Y$ is an inflation \iff the map $\Hom_\C(f,I): \Hom_\C(Y,I)\la \Hom_\C(X,I)$ is an epimorphism of abelian groups for every injective object $I$.

\item If $[(X_\alpha)_{\alpha <\lambda},(i_{\alpha\beta})_{\alpha <\beta <\lambda}]$ is a $\lambda$-sequence of inflations having a composition $X_0 \la \colim_{\alpha < \lambda} X_\alpha$ in $\C$, then the composition is an inflation.

\item If $(0\to X_i\overset{f_i}\la Y_i\overset{g_i}\la Z_i\to 0)_{i\in I}$ is a family of conflations having a coproduct, then
$$
0 \la
\coprod_{i \in I} X_i \xrightarrow{\coprod f_i}
\coprod_{i \in I} Y_i \xrightarrow{\coprod g_i}
\coprod_{i \in I} Z_i \la
0
$$
is a conflation.
\end{enumerate}
\end{prop}

\begin{proof}
(1) The ``only if'' part of the statement is clear. For the converse, fix any inflation $j:X\infl I$, with $I$ injective. Using the assumption that $\Hom_\C(f,I): \Hom_\C(Y,I)\la \Hom_\C(X,I)$ is surjective, choose a morphism $g:Y\la I$ such that $gf=j$. By the dual of Lemma~\ref{lem:adding a deflation}, the morphism $\smallpmatrix{f \\ gf}: X\la Y\oplus I$ is an inflation. Then the dual of Lemma \ref{lem:(p,pu) deflation implies f deflation} tells us that $f$ is an inflation.

(2) Suppose that $[(X_\alpha )_{\alpha<\lambda},(i_{\alpha\beta}:X_\alpha\la X_\beta )_{\alpha<\beta <\lambda}]$ is a $\lambda$-sequence in $\C$ having a composition $f: X_0 \la X$. If $I$ is an injective object of $\C$, then
$$
\left[
\big( \Hom_\C(X_\alpha,I) \big)_{\alpha <\lambda},
\big( \Hom_\C(i_{\alpha\beta},I): \Hom_\C(X_\beta,I) \la \Hom_\C(X_\alpha,I) \big)_{\alpha <\beta <\lambda}
\right]
$$
is a continuous well-ordered inverse system of epimorphisms of abelian groups. Moreover, the limit of the inverse system is $\Hom_\C(X,I)$ and
$$ \Hom_\C(f,I): \Hom_\C(X,I) \la \Hom_\C(X_0,I) $$
is the corresponding limit morphism. It is not difficult to prove that $\Hom_\C(f,I)$ is surjective. Namely, given $x_0 \in \Hom_\C(X_0,I)$, one inductively constructs a sequence $(x_\alpha)_{\alpha<\lambda}$ \st $x_\alpha \in \Hom_\C(X_\alpha,I)$ and $x_\alpha = x_\beta \circ i_{\alpha\beta}$ for each $\alpha<\beta<\lambda$. The colimit property in $\C$ gives us $x = \colim_{\alpha<\lambda} x_\alpha \in \Hom_\C(X,I)$ \st $x_0 = x \circ f$. Finally, assertion (1) implies that $f: X_0 \la X$ is an inflation.

(3) One readily sees that $\coprod g_i:\coprod Y_i\la \coprod Z_i$ is the cokernel map of $\coprod f_i$. The proof is whence reduced to check that $\coprod f_i$ is an inflation. But this is a direct consequence of assertion (1) bearing in mind that $\Hom_\C(-,I)$ takes coproducts in $\C$ to products of abelian groups, and products of abelian groups are exact.
\end{proof}

\bibliographystyle{abbrv}
\bibliography{exact_catg_bib}

\end{document}